\documentclass[11pt]{article}
\usepackage{amsmath}
\usepackage{amssymb}
\usepackage{amsthm}

\newtheorem{theorem}{Theorem}[section]
\newtheorem{lemma}[theorem]{Lemma}
\newtheorem{corollary}[theorem]{Corollary}
\newtheorem{proposition}[theorem]{Proposition}
\begin{document}

\def\C{{\mathbb C}}
\def\N{{\mathbb N}}
\def\Z{{\mathbb Z}}
\def\R{{\mathbb R}}
\def\PP{\cal P}
\def\p{\rho}
\def\phi{\varphi}
\def\ee{\epsilon}
\def\ll{\lambda}
\def\l{\lambda}
\def\a{\alpha}
\def\bb{\beta}
\def\D{\Delta}
\def\dd{\delta}
\def\g{\gamma}
\def\rk{\text{\rm rk}\,}
\def\dim{\text{\rm dim}\,}
\def\ker{\text{\rm ker}\,}
\def\square{\vrule height6pt width6pt depth 0pt}
\def\epsilon{\varepsilon}
\def\phi{\varphi}
\def\kappa{\varkappa}
\def\strl#1{\mathop{\hbox{$\,\leftarrow\,$}}\limits^{#1}}
\def\lal{{\Lambda\!}^L}
\def\lar{{\Lambda\!}^R}

\vskip1cm

\title{The proof of the Kontsevich periodicity conjecture on noncommutative birational transformations}

\author{
 Natalia Iyudu, Stanislav Shkarin}

\date{}

\maketitle

\begin{abstract} For an arbitrary associative unital ring $R$, let $J_1$ and $J_2$ be the following noncommutative birational partly defined involutions on the set $M_3(R)$ of $3\times 3$ matrices over $R$: $J_1(M)=M^{-1}$ (the usual matrix inverse) and $J_2(M)_{jk}=(M_{kj})^{-1}\,$
 (the transpose of the Hadamard inverse).

 We prove the following surprising conjecture by Kontsevich (1996) saying that $(J_2\circ J_1)^3$ is the identity map modulo the ${\rm Diag}_{L} \times \rm{Diag}_R$ action $(D_1,D_2)(M)=D_1^{-1}MD_2$ of pairs of invertible diagonal matrices.

 That is, we show that for each $M$ in the domain where $(J_2\circ J_1)^3$ is defined, there are invertible diagonal $3\times 3$ matrices $D_1=D_1(M)$ and $D_2=D_2(M)$ such that $(J_2\circ J_1)^3(M)=D_1^{-1}MD_2$.
\end{abstract}


\normalsize

\section{Introduction \label{s1}}\rm

The conjecture we deal with appeared in around 1996 from the idea to generalise the following commutative picture in the noncommutative setting.
Consider for a fixed ${n,m}$ configurations of $n+m$ points ${\cal C}_{n,m}$ on the projective space $\mathbb P^{n-1}$, up to an action of $\rm{Gl}_n \times \rm{Diag }_{n+m}$ \cite{Dolg}.
Symmetric group $S_{n+m}$ acts by permuting points (columns of $n \times (n+m)$ matrices). We can normalise a $n \times (n+m)$  matrix, representing a point, by making the beginning part into the identity matrix,
that is we get space of $n \times m$ matrices, up to $\rm{Diag }_{n}\times \rm{Diag }_{m}$ action, where the Cremona tranform

 $$Cr: \alpha_{ij} \to \alpha_{ij}^{-1},\,\, 1 \leq i \leq n, 1 \leq j \leq m $$
is defined.
The symmetric group $S_{n+m}$ together with the Cremona involution $Cr$ generate a birational action  of the Coxeter group $E_{n,m}$, on ${\cal C}_{n,m}$, where the usual relations holds:

$$Cr \circ \sigma_i=  \sigma_i \circ Cr, \,\, i \neq n $$
and
$$ (Cr \circ \sigma_n)^3=id.$$

  When one tries to generalise  this picture for noncommutative variables, some identities of the Coxeter type hold. The mysterious identity in the Kontsevich Conjecture should be the one of the similar  nature, but it turned out to be quite difficult to prove.

At present very little is known about the 'unglued' noncommutative picture behind the familiar commutative world. The geometry which may exist in the noncommutative setting is poorly understood. This paper constitute an attempt to find an appropriate combinatorial tool to deal with noncommutative identities describing these 'unusual' geometries.

In the notes of the 2011 Arbeitstagung talk on 'Noncommutative identities' (arxiv1109.2469) the conjecture was formulated as follows.

Let $(M_{ij})_{1 \leq i,j \leq 3}$ be a matrix, whose entries are $9=3 \times 3$ independent noncommutative variables. Let us consider three 'birational involutions'

\vspace{5mm}
$$
\begin{array}{ll}I_1: \,\, M \to M^{-1}&\text{:the matrix inverse;}
\\
I_2: \,\, M_{ij} \to (M_{ij})^{-1}, \,\, \forall i,j&\text{:inverting the matrix entries;}
\\
I_3: \,\, M \to M^t&\text{:the transpose.}
\end{array}
$$
\vspace{3mm}

The composition $I_3 \circ I_2 \circ I_1 $ commutes with the multiplication on the left
and on the right by diagonal $3 \times 3$ matrices. We can factorize it by the
action of ${\rm Diag}_{L} \times \rm{Diag}_{R}$  and get only 4 independent variables, setting
e.g. $ M_{ij} = 1 $ for $\min\{i,j\} = 1$.

\vspace{5mm}

{\bf Conjecture}. (M.Kontsevich)( \cite{K},  Conjecture 1 in section 3)
{\it
The transformation
$(I_3 \circ I_2 \circ I_1)^3$ is equal to the identity modulo
${\rm Diag}_{L} \times \rm{Diag}_R$ - action. In other words,
there exists two diagonal $3 \times 3$ matrices $D_L(M)$ and $D_R(M)$, whose entries are noncommutative rational functions in $9$ variables $M_{ij}$, such that

$$ (I_3 \circ I_2 \circ I_1)^3(M)=D_L(M) \, M \, D_R(M).$$
}

\vspace{5mm}

What we can say for sure, after proving the conjecture,
 is that the group $S_3 \rtimes_{\sigma} S_3 \times S_3$
(consisting of 216 elements), where $\sigma$ is a flip automorphism: $\sigma(a,b)=(b,a)$, has a faithful representation by noncommutative birational transformations of 9  variables.

 In more general situations this kind of noncommutative birational transformations can provide a noncommutative discrete integrable system, as it is explained in \cite {K}.
In \cite {K} there were formulated also several other conjectures of the similar spirit, which have to do with noncommutative birational transformations. Some of them are solved, or partially solved, for example, in \cite {Re}, \cite {RK}.

The object consisting  of 'noncommutative rational functions' on certain set of generators
$X=\{x_1,...,x_n\}$ has been considered in \cite{R}. We denote here by ${\cal R}(X)$ the division ring of free
noncommutative (but associative) rational expressions on alphabet $X$.
The elements of this ring serve as matrix elements for presentations of the operators $I_1, I_2, I_3 $ and its iterations. The difficulty of the problem based on the fact that no useful {\it normal form} of an element of that 'free division ring' $\cal R$ is available.
It was shown in \cite{CR} that in the rigourously defined there 'free field' the {\it word problem} is solvable, but nothing is known about the  {\it conjugacy problem}.
Anyway, an appropriate (sometimes just lucky) choice of representatives for certain elements in $\cal R$ allowed to proceed with the prove of the conjecture, which requires to check both conjugacy and equality.

Our proof of the conjecture allows to formulate a slightly stronger result, which says that the conjecture is true not only over a universal free object, where some or all elements are invertible (for example, over the free field), but over {\it any} algebra, where the {\it finite set} of rational expressions on generators  are defined and invertible.

This fact has to do with another interesting phenomenon featuring in this situation:
{\it the confinement of singularities} \footnote{we would like to thank A.Veselov for drawing our attention on some activities around study of this phenomenon}
for iterations of this particular rational map. Nothing like this holds, for example, for 4-dimensional case. The property that singularities of rational expressions appearing from iterations of the map are stabilising, considered to be a strong form of {\it integrability}.

From now on we assume that
$R$ is an (arbitrary) unital associative ring and $R^*$ is the group of invertible elements of $R$.
$M_n(R)$ is obviously a unital ring with respect to both the usual matrix multiplication $*$ and the Hadamard (=componentwise) matrix multiplication $\star$.

Our definitions of   $J_1$ and  $J_2$ in terms of these two multiplications  means  that  $J_1$ maps $M$ into its $*$-inverse, while $J_2$ maps $M$ into the transpose of its $\star$-inverse.

Thus the domain of the definition of $J_1$, ${\rm dom}\,(J_1)$ is the set $M_n^*(R)$ of all $*$-invertible matrices in $M_n(R)$, while the domain of the definition of $J_2$, ${\rm dom}(J_2)$ is the set $M_n^\star$ of all $\star$-invertible matrices in $M_n(R)$ (=the matrices with invertible entries). Denote
$$
J:=J_2\circ J_1.
$$
We always assume that for a composition of two maps $f$ and $g$, the natural domain of $g\circ f$ is ${\rm dom}(g\circ f)=\{x\in{\rm dom}(f):f(x)\in{\rm dom}(g)\}$. In particular,
$$
{\rm dom}(J)=\{M\in M_n^*(R): M^{-1}\in M^\star(R)\},
$$
${\rm dom}(J^2)=\{M\in{\rm dom}(J):J(M)\in{\rm dom}(J)\}$ etc.
As both $J_1$ and $J_2$ are involutions, $J^{-1}=J_1\circ J_2$ and
$$
{\rm dom}(J^{-1})=\{M\in M_n^\star(R): J_2(M)\in M_n^*(R)\}.
$$

 As above ${\rm Diag}_L \times {\rm Diag}_R $  be the group of pairs $(D_1,D_2)$ of invertible diagonal matrices
  which acts on $M_n(R)$ according to the rule
$$
(D_1,D_2)(M)=D_1^{-1}MD_2.
$$

Further we will denote this group in a shorter way by ${\cal D}_{LR}$.

For $A,B\in M_n(R)$, we write
$$
A\sim B\ \ \text{if $A$ and $B$ are in the same ${\cal D}_{LR}$-orbit.}
$$
In other words, $A\sim B$ if and only if $B$ can be obtained from $A$ by multiplying on both sides by invertible diagonal matrices.

It is easy to see that for any $D_1$ and $D_2$, $\,\,J_k(D_1^{-1}MD_2)=D_2^{-1}J_k(M)D_1$ for $k=1,2$.
Thus both $J_1$ and $J_2$ act on ${\cal D}_{LR}$-orbits.
Hence we have even more then needed, to say that $J$ also acts on orbits, in fact the following is true:

$$J(D_1^{-1}MD_2)=D_1^{-1}J(M)D_2.$$

So, we have
\begin{equation}\label{orb1}
\text{if $A\sim B$, then $J_1(A)\sim J_1(B)$, $J_2(A)\sim J_2(B)$ and $J(A)\sim J(B)$}
\end{equation}

(Note, that this statement supposed to take into account domains of the maps, namely it is meant that if $A\sim B$, then either both $A$ and $B$ are not in the domain of the relevant map or both are in its domain and the map sends them to the same ${\cal D}_{LR}$-orbit).

Now we can state the main result.

\begin{theorem}\label{main1}
Let $R$ be an arbitrary unital
associative ring, and $J=J_2 \circ J_1$  acts on  $M_3(R)$, then   for every $M\in {\rm dom}(J^3)$,  $J^3(M)\sim M$.
\end{theorem}

In other words, for each $M\in M_3(R)$ in the domain of $J^3$, there is a pair $(D_1,D_2)$ of invertible diagonal $3\times 3$ matrices over $R$ such that $J^3(M)=D_1^{-1}MD_2$.

The domain  of $J^3$ is described as follows.

\begin{proposition}\label{domaAA} When acting on $M_3(R)$, the domain of $J^n$ does not depend on $n$ for $n\geq 2$. For $n\geq 2$, ${\rm dom}(J^n)={\rm dom}(J)\cap {\rm dom}(J^{-1})$ and $M$ belongs to ${\rm dom}(J^n)$ if and only if all square submatrices of $M$ and of $J_2(M)$ are invertible
 \end{proposition}

 The latter condition amounts to the invertibility of finite number of 'noncommutative determinants'
 of submatrices of $M$ and $J_2(M)$.

The rest of the paper is mainly devoted to the proof of Theorem~\ref{main1}. Our proof is constructive. In particular, it provides (after a suitable reduction of the number of free parameters) an explicit formula for the relevant maps and diagonal matrices.

\section{Preliminary facts and $2\times 2$ case}

In this elementary section we deal with the baby version of the Conjecture. Namely, we look at the case of $2\times 2$ matrices. The purpose of this section is two-fold. First, it provides a glimpse into the technique used. Second, the formulas obtained in this section will be used later in the proof of Theorem~\ref{main1}.

The following lemma is known and can, for example, be extracted from \cite{R}. We prove it for the reader's convenience.

\begin{lemma}\label{2t2inv} Let
\begin{equation*}
A=\left(\begin{array}{cc}
a&b\\
c&d
\end{array}
\right)\in M_2(R).
\end{equation*}
Then
\begin{equation}\label{first2}
A\in M_2^\star(R)\ \Longrightarrow\ \bigl(A\in M_2^*(R)\iff J_2(A)\in M_2^*(R)\iff db^{-1}-ca^{-1}\in R^*\bigr).
\end{equation}
Furthermore, if $A\in M_2^*(R)\cap M_2^\star(R)$, then $a-bd^{-1}c$, $c-db^{-1}a$, $b-ac^{-1}d$ and $d-ca^{-1}b$ are invertible in $R$ and $A^{-1}$ is given by the formula
\begin{equation}\label{inva}
A^{-1}=\left(\begin{array}{cc}
(a-bd^{-1}c)^{-1}&(c-db^{-1}a)^{-1}\\
(b-ac^{-1}d)^{-1}&(d-ca^{-1}b)^{-1}
\end{array}
\right).
\end{equation}
And finally, the following implication holds:
\begin{equation}\label{second2}
A\in M_2^*(R)\ \Longrightarrow\ \bigl(A\in M_2^\star(R)\iff A^{-1}\in M_2^\star(R)\bigr).
\end{equation}
\end{lemma}

\begin{proof} Assume that $A\in M_2^\star(R)$. That is, $a,b,c,d\in R^*$. First, we prove that $A$ is invertible if and only if $db^{-1}-ca^{-1}\in R^*$ and that (\ref{inva}) holds. If $db^{-1}-ca^{-1}\in R^*$, then the obvious equalities
$$
\begin{array}{ll}
d-ca^{-1}b=(db^{-1}-ca^{-1})b,&c-db^{-1}a=-(db^{-1}-ca^{-1})a,\\
b-ac^{-1}d=-ac^{-1}(db^{-1}-ca^{-1})b,&a-bd^{-1}c=bd^{-1}(db^{-1}-ca^{-1})a
\end{array}
$$
imply that $a-bd^{-1}c$, $c-db^{-1}a$, $b-ac^{-1}d$ and $d-ca^{-1}b$ are invertible in $R$. This makes the matrix $B$ in the right-hand side of (\ref{inva}) well-defined. It is a straightforward exercise to verify that $AB=BA=I$. We shall show that the top left entry of $AB$ equals to $1$; the other  equalities are either trivial or verified in a similar manner. We have to check that $a(a-bd^{-1}c)^{-1}+b(b-ac^{-1}d)^{-1}=1$. Denoting $x=bd^{-1}ca^{-1}$, we have
\begin{align*}
&a(a-bd^{-1}c)^{-1}+b(b-ac^{-1}d)^{-1}=(1-bd^{-1}ca^{-1})^{-1}+(1-ac^{-1}db^{-1})^{-1}
\\
&\qquad\qquad=(1-x)^{-1}+(1-x^{-1})^{-1}=(1-x)^{-1}-x(1-x)^{-1}=(1-x)(1-x)^{-1}=1.
\end{align*}
The equality $AB=BA=I$ means that $A$ is invertible and that (\ref{inva}) holds.

Assume now that $A$ is invertible and let
\begin{equation*}
A^{-1}=\left(\begin{array}{cc}
s&t\\
u&v
\end{array}
\right)\ \Longrightarrow \
\left(\begin{array}{cc}
a&b\\
c&d
\end{array}
\right)\left(\begin{array}{cc}
s&t\\
u&v
\end{array}
\right)=
\left(\begin{array}{cc}
s&t\\
u&v
\end{array}
\right)\left(\begin{array}{cc}
a&b\\
c&d
\end{array}
\right)=\left(\begin{array}{cc}
1&0\\
0&1
\end{array}
\right).
\end{equation*}
In particular, $as+bu=sa+tc=1$ and $cs+du=sb+td=0$. Multiplying $cs+du=0$ by $bd^{-1}$ on the left and subtracting the result from $as+bu=1$, we get $(a-bd^{-1}c)s=1$. Multiplying $sb+td=0$ by $d^{-1}c$ on the right and subtracting the result from $sa+tc=1$, we get $s(a-bd^{-1}c)=1$. Thus $a-bd^{-1}c$ is invertible with the inverse being $s$. Since $db^{-1}-ca^{-1}=db^{-1}(a-bd^{-1}c)a^{-1}$ and $a,b,d,a-bd^{-1}c\in R^*$, $db^{-1}-ca^{-1}$ is invertible. Thus in the case $A\in M_2^\star(R)$, $A\in M_2^*(R)\iff db^{-1}-ca^{-1}\in R^*$ and (\ref{inva}) holds. By the just verified equivalence applied to $J_2(A)$, $J_2(A)\in M_2^*(R)\iff d^{-1}c-b^{-1}a\in R^*$. Since $-d(d^{-1}c-b^{-1}a)a^{-1}=db^{-1}-ca^{-1}$, $d^{-1}c-b^{-1}a\in R^*\iff db^{-1}-ca^{-1}\in R^*$. These observations complete the proof of
(\ref{first2}).

It remains to verify the implication (\ref{second2}). Note that (\ref{inva}) implies that
$$
C\in M_2^*(R)\cap M_2^\star(R)\ \Longrightarrow C^{-1}\in M_2^*(R)\cap M_2^\star(R).
$$
Indeed, the matrix entries in (\ref{inva}) are invertible.
Applying that to $A$ and $A^{-1}$, makes (\ref{second2}) obvious.
\end{proof}

\begin{proposition}\label{2b2case} For $J$ acting on $M_2(R)$, ${\rm dom}(J)={\rm dom}(J^{-1})=M_2^*(R)\cap M_2^\star(R)$, and $J$ acts bijectively on ${\rm dom}(J)$  according to the formula
\begin{equation}\label{zzz}
J(A)=\left(\begin{array}{cc}
a-bd^{-1}c&b-ac^{-1}d\\
c-db^{-1}a&d-ca^{-1}b
\end{array}
\right),\ \ \text{where}\ \
A=\left(\begin{array}{cc}
a&b\\
c&d
\end{array}
\right)\in M^\star_2(R)\cap M_2^*(R).
\end{equation}
Furthermore, $J(A)\sim J^{-1}(A)$ for each $A\in {\rm dom}(J)$, and $A\in {\rm dom}(J)={\rm dom}(J^{-1})$ if and only if $a,b,c,d,db^{-1}-ca^{-1}\in R^*$ and
\end{proposition}

{\bf Remark } By Proposition~\ref{2b2case}, $J^{-1}(A)\sim J(A)$ whenever $A\in M_2(R)$ belongs to ${\rm dom} (J)={\rm dom} (J^{-1})$. Hence $J^2(A)\sim A$ whenever $A\in {\rm dom} (J)$. Thus the conclusion of the Conjecture holds for $2\times 2$ matrices if we take the second power of $J$ instead of the third.

\begin{proof}[Proof of Proposition~$\ref{2b2case}$] Lemma~\ref{2t2inv} implies that $A\in{\rm dom} (J)$ if and only if $A\in {\rm dom} (J^{-1})$ if and only if $A\in M_2^*(R)\cap M_2^*(R)$ if and only if the 5 elements $a$, $b$, $c$, $d$ and $db^{-1}-ca^{-1}$ are invertible in $R$. Since ${\rm dom} (J)={\rm dom} (J^{-1})$, $J$ maps ${\rm dom} (J)$ bijectively onto itself. Furthermore, (\ref{inva}) yields (\ref{zzz}), which can be easily rewritten as
\begin{align*}
J(A)&=\left(\begin{array}{cc}
(ac^{-1}d-b)d^{-1}c&(ac^{-1}d-b)(-b^{-1})b\\
(db^{-1}a-c)(-c^{-1})c&(db^{-1}a-c)a^{-1}b
\end{array}
\right)
\\
&=\left(\begin{array}{cc}
ac^{-1}d-b&0\\
0&db^{-1}a-c
\end{array}
\right)
\left(\begin{array}{cc}
d^{-1}&-b^{-1}\\
-c^{-1}&a^{-1}
\end{array}
\right)
\left(\begin{array}{cc}
c&0\\
0&b
\end{array}
\right).
\end{align*}
Since $J=J_2\circ J_1$ and $J^{-1}=J_1\circ J_2$, we have $J^{-1}=J_2\circ J\circ J_2$. Using this observation together with the above, one easily gets
\begin{align*}
J^{-1}(A)&=\left(\begin{array}{cc}
bd^{-1}(a^{-1}bd^{-1}-c^{-1})^{-1}&b(-b^{-1})(d^{-1}ca^{-1}-b^{-1})^{-1}\\
c(-c^{-1})(a^{-1}bd^{-1}-c^{-1})^{-1}&ca^{-1}(d^{-1}ca^{-1}-b^{-1})^{-1}
\end{array}
\right)
\\
&=\left(\begin{array}{cc}
b&0\\
0&c
\end{array}
\right)
\left(\begin{array}{cc}
d^{-1}&-b^{-1}\\
-c^{-1}&a^{-1}
\end{array}
\right)
\left(\begin{array}{cc}
(a^{-1}bd^{-1}-c^{-1})^{-1}&0\\
0&(d^{-1}ca^{-1}-b^{-1})^{-1}
\end{array}
\right).
\end{align*}
So, we can see that $J(A)\sim J^{-1}(A)$. \end{proof}

Note that the above two displays provide explicit formulas for two invertible diagonal matrices facilitating the relation $J^{-1}(A)\sim J(A)$ in the case $n=2$.

\section{The transformation $J$ on the level of orbits: definition of $\Phi$}

In this section we shall explain a reduction of the Conjecture to a 4-parameter problem, which is outlined in \cite{K}. The reduction itself works for matrices of arbitrary size, so we present it
as a reduction of $n \times n$ matrices to $(n-1) \times (n-1)$. It turns out that for any $M\in M_n^\star(R)$, there is a (non-unique if $R$ is noncommutative) matrix $A\sim M$ such that
$$
A\in \widehat M_n(R)=\bigl\{A=\{a_{j,k}\}\in M_n(R):a_{1,k}=a_{k,1}=1\ \ \text{for $1\leq k\leq n$}\bigr\}.
$$
That is, the entries in the first row and the first column of $A$ all equal $1$. More specifically, consider the (noncommutative) rational maps $\lal,\lar: M_n^\star(R)\to \widehat{M}_n(R)$ defined by the formulas

\begin{equation}\label{lal}
\lal(A)_{j,k}=a_{1,1}a_{j,1}^{-1}a_{j,k}a_{1,k}^{-1}
\end{equation}

\begin{equation}\label{lar}
\lar(A)_{j,k}=a_{j,1}^{-1}a_{j,k}a_{1,k}^{-1}a_{1,1}
\end{equation}

where $A=\{a_{j,k}\}$.

\begin{lemma}\label{4pA}
Let $A=\{a_{j,k}\}\in M^\star_n(R)$. Then $\lal(A)$ is the unique matrix $B\in \widehat{M}_n(R)$ such that there exist invertible diagonal $n\times n$ matrices $D_1$ and $D_2$ with the top left entry of $D_1$ being $1$ for which $B=D_1^{-1}AD_2$. Similarly, $\lar(A)$ is the unique matrix $C\in \widehat{M}_3(R)$ such that there exist invertible diagonal $n\times n$ matrices $D_3$ and $D_4$ with the top left entry of $D_4$ being $1$ for which $C=D_3^{-1}AD_4$. In particular,
\begin{equation}\label{lamlam}
\lal(A)\sim \lar(A)\sim A\ \ \text{for each}\ \ A\in M_3^\star(R).
\end{equation}
\end{lemma}

\begin{proof} For $1\leq s\leq 4$, let $D_s\in M_n(R)$ be diagonal with the entries $d_{s,j}\in R^*$ for $1\leq j\leq n$ on the main diagonal with the extra assumption $d_{1,1}=d_{4,1}=1$. Then the entries $b_{j,k}$ of $B=D_1^{-1}AD_2$ are given by $b_{j,k}=d_{1,j}^{-1}a_{j,k}d_{2,k}$. The condition $B\in \widehat{M}_n(R)$ is equivalent to $b_{1,k}=b_{k,1}=1$. Since $d_{1,1}=1$, the equations $b_{1,k}=1$ are equivalent to $d_{2,k}=a_{1,k}^{-1}$ for $1\leq k\leq n$. Now the  equations $b_{k,1}=1$ for $2\leq k\leq n$ hold if and only if $d_{1,k}=a_{k,1}a_{1,1}^{-1}$ for $2\leq k\leq n$. This
uniquely determines $D_1$ and $D_2$ with $d_{1,1}=1$ for which $B=D_1^{-1}AD_2\in \widehat{M}_n(R)$. Plugging this data back into $B=D_1^{-1}AD_2$, we find that $B=\lal(A)$ as defined in (\ref{lal}).

Next, the entries $c_{j,k}$ of $C=D_3^{-1}AD_4$ are given by $c_{j,k}=d_{3,j}^{-1}a_{j,k}d_{4,k}$. The condition $C\in \widehat{M}_3(R)$ is equivalent to $c_{1,k}=c_{k,1}=1$. Since $d_{4,1}=1$, the equations $c_{k,1}=1$ read $d_{3,k}=a_{k,1}$ for $1\leq k\leq n$. Now the  equations $c_{1,k}=1$ for $2\leq k\leq n$ hold if and only if $d_{4,k}=a_{1,k}^{-1}a_{1,1}$. This uniquely determines $D_3$ and $D_4$ with $d_{4,1}=1$ for which $B=D_3^{-1}AD_4\in \widehat{M}_n(R)$. Plugging this data back into $C=D_3^{-1}AD_4$, we find that $B=\lar(A)$ as defined in (\ref{lar}). Finally, (\ref{lamlam}) now follows directly from the definition of the relation $\sim$.
\end{proof}

The relation (\ref{lamlam}) guarantees that every $A\in M_n^\star(R)$ has a member of ${\widehat M}_n(R)$ in its ${\cal D}_{LR}$-orbit. When $R$ is noncommutative, one orbit may contain more than one matrix from ${\widehat M}_n(R)$.

We are going to characterise in the next lemma the matrices from ${\widehat M}_n(R)$, which  are at the same ${\cal D}_{LR}$-orbit.

\begin{lemma}\label{nonu} (Conjugacy)
Let $(D_1,D_2)\in {\cal D}_{LR}$ and $A,B\in \widehat{M}_n(R)$. Then the equality $B=D_1^{-1}AD_2$ holds if and only if there is $x\in R^*$ such that $D_1=D_2=xI$, where $I$ is the $n\times n$ identity matrix and $B=x^{-1}Ax$, where $M_n(R)$ is equipped with the obvious $R$-bimodule structure. In particular, $A\sim B$ if and only if there is $x\in R^*$ such that $B=x^{-1}Ax$.
\end{lemma}

\begin{proof} Let $x_1,\dots,x_n$ be the diagonal entries of $D_1$, $y_1,\dots,y_n$ be the diagonal entries of $D_2$,
$A=\{a_{j,k}\}$ and $B=\{b_{j,k}\}$. Since $A,B\in \widehat{M}_n(R)$, $a_{1,k}=b_{1,k}=a_{k,1}=b_{k,1}=1$  for $1\leq k\leq n$. Now looking at the first row and the first column of the matrix equality $B=D_1^{-1}AD_2$, we see that $x_1^{-1}y_k=x_k^{-1}y_1=1$ for $1\leq k\leq n$. Denoting $x=x_1$, we see that these equalities can only hold if $x_k=y_k=x$ for $1\leq k\leq n$. Thus $D_1=D_2=xI$. Now the equality $B=D_1^{-1}AD_2$ reads $B=x^{-1}Ax$.
\end{proof}

Now we will choose a 4-parameter  map $\Phi$, which is coincide with $J$ on the level of ${\cal D}_{LR}$-orbits, using the map $\lal$.

{\bf Definition} We define the noncommutative rational map $\Phi$ as acting on ${\widehat M}_n(R)$ according to the formula
\begin{equation*}
\Phi(A)=(J_2\circ \lal\circ J_1)(A)=J_2(\lal(A^{-1}))\ \ \text{with the domain}\ \
{\rm dom}(\Phi)={\rm dom} (J)\cap \widehat M_n(R)\cap M_n^\star(R).
\end{equation*}
Using (\ref{orb1}) and (\ref{lamlam}), we indeed  see that
\begin{equation}\label{phij}
\text{$\Phi(A)\sim J(A)$ for every $A\in{\rm dom}(\Phi)$.}
\end{equation}

Note that, although the map $\lal$ is certainly non-injective, the map $\Phi$ turns out to be injective and we can give an explicit formula for its inverse, we do it at the next section.

\section{The inverses of $\Phi$}

\begin{proposition}\label{phiinve} The map $\Phi'$ acting on ${\widehat M}_n(R)$ according to the formula
\begin{equation*}
\begin{array}{l}\Phi'(A)=(\lar\circ J_1\circ J_2)(A)=\lar(J^{-1}(A))\ \ \text{with the domain}\\
{\rm dom}(\Phi')=\widehat M_n(R)\cap \{M\in {\rm dom} (J^{-1}):J^{-1}(M)\in M_n^\star(R)\}
\end{array}
\end{equation*}
is the inverse of $\Phi$. That is, $\Phi$ maps ${\rm dom}(\Phi)$ bijectively onto ${\rm dom}(\Phi')$, $\Phi'$ maps ${\rm dom}(\Phi')$ bijectively onto ${\rm dom}(\Phi)$ and $\Phi\circ\Phi'$, $\Phi'\circ\Phi$ are identity maps on ${\rm dom}(\Phi')$ and ${\rm dom}(\Phi)$ respectively.
\end{proposition}

\begin{proof} It suffices to show that for every $A\in {\rm dom}(\Phi)$ and every $B\in {\rm dom}(\Phi')$,
\begin{align}
\text{$\Phi(A)\in {\rm dom}(\Phi')$ and $\Phi'(\Phi(A))=A$;}\label{phipsi1}
\\
\text{$\Phi'(B)\in {\rm dom}(\Phi)$ and $\Phi(\Phi'(B))=B$}.\label{phipsi2}
\end{align}

Let $A\in {\rm dom}(\Phi)$. Then $A\in {\rm dom} (J)$, $A\in \widehat M_n(R)$ and $A\in M_n^\star(R)$. Since $A\in {\rm dom} (J)$, $A\in M^*_n(R)$ and $A^{-1}\in M_n^\star(R)$. By definition $\Phi(A)=J_2(\lal(A^{-1}))$. Obviously, $\Phi(A)\in \widehat M_n(R)\cap M_n^\star(R)$. By (\ref{lamlam}), $J_2(\Phi(A))=\lal(A^{-1})\sim A^{-1}$ and therefore $J_2(\Phi(A))\in M_n^*(R)$. Hence $\Phi(A)\in {\rm dom}(J^{-1})$. By (\ref{phij}), $J^{-1}(\Phi(A))\sim A$. Since $A\in M_n^\star(R)$,
$J^{-1}(\Phi(A))\in M_n^\star(R)$  and therefore $\Phi(A)\in {\rm dom}(\Phi')$. Furthermore,
$$
\Phi'(\Phi(A))=\lar(J_1(J_2(J_2(\lal(J_1(A))))))=\lar(J_1(\lal(J_1(A)))).
$$
By Lemma~\ref{4pA}, there exist invertible diagonal matrices $D_k$ for $1\leq k\leq 4$ such that the top-left entries of $D_1$ and $D_4$ equal $1$ and
$$
\Phi'(\Phi(A))=D_3^{-1}(D_1^{-1}A^{-1}D_2)^{-1}D_4=(D_2D_3)^{-1}A(D_1D_4).
$$
Since both $A$ and $\Phi'(\Phi(A))$ belong to $\widehat{M}_n(R)$, Lemma~\ref{nonu} implies that all diagonal entries of the diagonal matrices $D_2D_3$ and $D_1D_4$ equal to the same $x\in R^*$. Since the top-left entry of $D_1D_4$ is $1$, we have $x=1$. Hence $D_2D_3=D_1D_4=I$. This together with the above display shows that $\Phi'(\Phi(A))=A$ and proves (\ref{phipsi1}). Let $B\in {\rm dom}(\Phi')$. Then $B\in \widehat M_n(R)\cap M_n^\star(R)$, $J_2(B)\in M_n^*(R)$ and $(J_2(B))^{-1}\in M_n^\star(R)$. By definition $\Phi'(B)=\lar(J_1(J_2(B)))$. By (\ref{orb1}), $\Phi'(B)\sim (J_2(B))^{-1}$ and therefore $\Phi'(B)\in \widehat{M}_n(R)\cap M_n^\star(R)\cap M_n^*(R)$. Since $(\Phi'(B))^{-1}\sim J_2(B)$, $\Phi'(B)^{-1}\in M_n^\star(R)$. Thus $\Phi'(B)\in{\rm dom}(\Phi)$. Furthermore,
$$
\Phi(\Phi'(B))=J_2(\lal(J_1(\lar(J_1(J_2(B)))))).
$$
By Lemma~\ref{4pA}, there exist invertible diagonal matrices $D_k$ for $1\leq k\leq 4$ such that the top-left entries of $D_1$ and $D_4$ equal $1$ and
$$
\Phi(\Phi'(B))=J_2(D_1^{-1}J_1(D_3^{-1}J_1(J_2(B))D_4)D_2).
$$
Since $J_k(\Delta_1^{-1}M\Delta_2)=\Delta_2^{-1}J_k(M)\Delta_1$ for $k\in\{1,2\}$ and $(\Delta_1,\Delta_2)\in{\cal D}_{LR}$ and $J_k$ are involutions, we arrive to $\Phi(\Phi'(B))=(D_3D_2)^{-1}B(D_4D_1)$. The same application of Lemma~\ref{nonu} as above yields $D_2D_3=D_1D_4=I$. Thus $\Phi(\Phi'(B))=B$ and proves (\ref{phipsi2}).
\end{proof}

From now on, we shall just write $\Phi^{-1}$ instead of $\Phi'$ defined in Proposition~\ref{phiinve}. Note that the relation (\ref{phij}) implies that
\begin{equation}\label{itephij}
\text{$\Phi^k(A)\sim J^k(A)$ for every $k\in\Z$ and every $A\in{\rm dom}(\Phi^k)$.}
\end{equation}

We introduce also another version of $\Phi^{-1}$, the map $\Psi$ defined as follows:
$$
\Psi=J_2\circ \Phi\circ J_2.
$$

The maps $\Psi$ and $\Phi^{-1}$ are the same at the level of orbits, but different at the level of elements.

\begin{proposition}\label{ph-1m} The maps $\Psi$ and $\Phi^{-1}$ are equivalent in the following sense:
they have the tsame domain, which they map bijectively onto ${\rm dom}(\Phi)$ and

$\Phi^{-1}(A)\sim \Psi(A,)$ for every $A \in
{\rm dom}(\Psi)={\rm dom}(\Phi^{-1})$.
\end{proposition}

\begin{proof}
 By our definition of $\Phi(A)$, we have $\Phi(A)\sim J(A)$, and hence also $\Phi(A)^{-1}\sim J(A)^{-1}$.
As we seen already $J^{-1}=J_2 \circ J \circ J_2$. This means that $\Psi(A) \sim J_2 \circ J \circ J_2(A).$

\end{proof}

The above proposition together with (\ref{orb1}) and (\ref{phij}) yields
\begin{equation}\label{psij}
\text{$\Psi(A)\sim \Phi^{-1}(A)\sim J^{-1}(A)$ for every $A\in{\rm dom}(\Phi^{-1})={\rm dom}(\Psi)$.}
\end{equation}

\subsection{Permutations and the symmetry lemma}

Here we consider the group of pairs $(P_1,P_2)$ of $n\times n$ permutation matrices (=the direct product of two copies of the group of $n\times n$ permutation matrices with respect to the usual matrix multiplication), denote it by  ${\cal P}_{LR}$. Clearly, ${\cal P}_{LR}$ is a subgroup of $M_n^*(R)\times M^*_n(R)$ isomorphic to $S_n\times S_n$ and ${\cal P}_{LR}$ acts on $M_n(R)$ according to the rule
$$
(P_1,P_2)(M)=P_1^{-1}MP_2.
$$
It is easy to see that $J_k(P_1^{-1}MP_2)=P_2^{-1}J_k(M)P_1$ for $k\in\{1,2\}$ and for each $(P_1,P_2)\in {\cal P}_{LR}$. Thus both $J_1$ and $J_2$ act on ${\cal P}_{LR}$-orbits. Moreover, $J(P_1^{-1}MP_2)=P_1^{-1}J(M)P_2$. That is, $J$  commutes with the multiplication (on the right or on the left) by a permutation matrix.

 The $\widehat{M}_n(R)$ is non-invariant with respect to the ${\cal P}_{LR}$-action. However it is invariant under the action of the subgroup ${\cal P}^0_{LR}$ of pairs of permutation matrices, which leaves the first basic vector invariant.

\begin{lemma}\label{kley} The sets ${\rm dom}(\Phi)$ and ${\rm dom}(\Psi)$ are stable under the ${\cal P}^0_{LR}$-action and $\Phi$ and $\Psi$ commute with the ${\cal P}^0_{LR}$-action. \end{lemma}

\begin{proof} The stability of ${\rm dom}(\Phi)$ and ${\rm dom}(\Psi)$ is straightforward. Let $A\in{\rm dom}(\Phi)$ and
$(P_1,P_2)\in {\cal P}^0_{LR}$. In order to show that $\Phi$ commutes with the ${\cal P}^0_{LR}$-action, we have to verify that $\Phi(P_1^{-1}AP_2)=P_1^{-1}\Phi(A)P_2$. By definition $\Phi(A)=J_2(\lal(A^{-1}))$. By Lemma~\ref{4pA}, there are unique invertible diagonal matrices $D_1$ and $D_2$ in $M_n(R)$ such that the upper-left entry of $D_1$ is $1$ and $\lal(A^{-1})=D_1^{-1}A^{-1}D_2$.
Then
$$
P_1^{-1}\Phi(A)P_2=P_1^{-1}J_2(\lal(A^{-1}))P_2=J_2(P_2^{-1}D_1^{-1}A^{-1}D_2P_1).
$$
Now observe that $D_2P_1=P_1\overline{D}_2$, where $\overline{D}_2$ is the diagonal matrix obtained from $D_2$ by the permutation of diagonal entries facilitated by $P_1$. Similarly, $D_1P_2=P_2\overline{D}_1$, where $\overline{D}_1$ is the diagonal matrix obtained from $D_1$ by the permutation provided by $P_2$. Thus
$$
P_1^{-1}\Phi(A)P_2=J_2(\overline{D}_1^{-1}P_2^{-1}A^{-1}P_1\overline{D}_2).
$$
Since $\overline{D}_1^{-1}P_2^{-1}A^{-1}P_1\overline{D}_2\in \widehat{M}_n(R)$  and the first diagonal entry of $\overline{D}_1$ is $1$ ($P_2$ leaves the first basic vector invariant), the uniqueness part of Lemma~\ref{4pA} yields
$\overline{D}_1^{-1}P_2^{-1}A^{-1}P_1\overline{D}_2=\lal(P_2^{-1}A^{-1}P_1)=\lal((P_1^{-1}AP_2)^{-1})$. Thus
$$
P_1^{-1}\Phi(A)P_2=J_2(\lal((P_1^{-1}AP_2)^{-1}))=\Phi(P_1^{-1}AP_2)
$$
and $\Phi$ commutes with the action of ${\cal P}^0_{LR}$.
Next, let $A\in{\rm dom}(\Psi)$ and $(P_1,P_2)\in {\cal P}^0_{LR}$. Then using the fact that $\Phi$ commutes with ${\cal P}^0_{LR}$-action, we get
\begin{align*}
P_1^{-1}\Psi(A)P_2&=P_1^{-1}J_2(\Phi(J_2(A))P_2)=J_2(P_2^{-1}\Phi(J_2(A))P_1)
\\
&=J_2(\Phi(P_2^{-1}J_2(A)P_1))=J_2(\Phi(J_2(P_1^{-1}AP_2)))=\Psi(P_1^{-1}AP_2)
\end{align*}
and $\Phi$ commutes with the action of ${\cal P}^0_{LR}$.
\end{proof}

In the $3\times 3$ case ${\cal P}^0_{LR}$ is the direct product of $2$ copies of $S_2=\Z_2$ and therefore is the Klein 4-group. We can list its elements. Let
$$
S=\left(\begin{array}{ccc}
1&0&0\\
0&0&1\\
0&1&0
\end{array}\right).
$$
Then in the case $n=3$, ${\cal P}^0_{LR}=\{(I,I),(I,S),(S,I),(S,S)\}=K$. Further on, we shall use the symbol $K$ to denote ${\cal P}^0_{LR}$ in the $3\times 3$ case. We state the following easy corollary of Lemma~\ref{kley} in the case $n=3$.

\begin{corollary}\label{kley1} (Symmetry Lemma) Let $n=3$ and $\Omega$ be an integer power of either $\Phi$ or $\Psi$.
For
\begin{equation*}
A=\left(\begin{array}{ccc}
1&1&1\\
1&a&b\\
1&c&d
\end{array}\right)\in {\rm dom}(\Omega)\subset {\widehat M}_3(R),\ \ \text{we denote}\ \
\Omega(A)=\left(\begin{array}{ccc}
1&1&1\\
1&\alpha&\beta\\
1&\gamma&\delta
\end{array}\right).
\end{equation*}
This makes $\alpha$, $\beta$, $\gamma$ and $\delta$ functions of $a$, $b$, $c$ and $d$. In particular,
$\alpha=f(a,b,c,d)$ for some map $f$. Then
$$
\beta=f(b,a,d,c),\ \ \gamma=f(c,d,a,b)\ \ \text{and}\ \ \delta=f(d,c,b,a).
$$
\end{corollary}

\begin{proof} Lemma~\ref{kley} implies that ${\rm dom}(\Omega)$ is invariant under the $K$-action and $\Omega$ commutes with the $K$-action. On the level of quadruples $(a,b,c,d)$ the elements of $K$ act by permutations in the following way. Obviously, $(I,I)$ corresponds to the identity permutation, $(I,S)$ corresponds to the permutation $(a,b,c,d)\mapsto (c,d,a,b)$, $(S,I)$ corresponds to the permutation $(a,b,c,d)\mapsto (b,a,d,c)$ and $(S,S)$ corresponds to the permutation $(a,b,c,d)\mapsto (d,c,b,a)$. The result immediately follows.
\end{proof}

\subsection{Reformulation of the conjecture and the integrability result (singularity confinement) in terms of 4-parameter maps}

For $3\times 3$ matrices, we can relate the domains of the powers of $\Phi$ and $J$ in a following way. Let
\begin{equation}\label{SS}
{\cal S}=\left\{M\in M_3(R):\begin{array}{l}
\text{all square submatrices of $M$ are}\\
\text{invertible and $J_2(M)$ is invertible}
\end{array}\right\}\ \ \text{and}\ \
\widehat{\cal S}={\cal S}\cap\widehat{M}_3(R).
\end{equation}

Note that the set $\cal S$ defined in (\ref{SS}) is invariant under the actions of ${\cal P}_{LR}$ and ${\cal D}_{LR}$. The following theorem describes the domains of the powers of $J$ and $\Phi$.

\begin{theorem}\label{DOMA} Let $n=3$ and $k\geq 2$. Then
\begin{equation}\label{doJPhi}
{\cal S}={\rm dom}(J)\cap {\rm dom}(J^{-1})={\rm dom}(J^k)\ \ \text{and}\ \
\widehat{{\cal S}}={\rm dom}(\Phi)\cap {\rm dom}(\Phi^{-1})={\rm dom}(\Phi^k).
\end{equation}
Furthermore, $J\bigr|_{\cal S}$ is a bijection from ${\cal S}$ onto itself and $\Phi\bigr|_{\widehat{\cal S}}$ is a bijection from $\widehat{{\cal S}}$ onto itself.
\end{theorem}

The effect that when acting on $3\times 3$ matrices, the domains of $J^k$ or $\Phi^k$ stop changing as $k$ grows is known as the {\it confinement of singularities}. If we iterate a generic rational map, even in the commutative case, the domain never stops shrinking. In the commutative case it corresponds to the fact that ever new irreducible factors occur in the denominator of the canonical expressions for the powers of our map. The singularities confinement is exactly the opposite of this effect: the domains of the powers stabilize.
The singularity confinement is considered as a strong integrability type property.

Now we can state the following theorem, which easily implies Theorem~\ref{main1}.

\begin{theorem}\label{main2} For every $A\in \widehat{\cal S}$, there is $x=x(A)\in R^*$ such that $\Phi^2(A)=x^{-1}\Phi^{-1}(A)x$. In particular, $\Phi^2(A)\sim \Phi^{-1}(A)$.
\end{theorem}

To see that Theorem~\ref{main2} implies Theorem~\ref{main1}, take $M\in{\rm dom}(J^3)$, where $J$ acts on $3\times 3$ matrices. By Theorem~\ref{DOMA}, $M\in{\cal S}$. Using (\ref{lamlam}), we can take $A\in \widehat{M}_3(R)$ such that $M\sim A$. Since $\cal S$ is obviously stable under the ${\cal D}_{LR}$-action, $A\in {\cal S}\cap \widehat{M}_3(R)=\widehat{\cal S}$. By Theorem~\ref{DOMA}, $\Phi(A)\in\widehat{S}$. By Theorem~\ref{main2}, applied to $\Phi(A)$, $\Phi^3(A)\sim A$. By (\ref{itephij}), $J^3(M)\sim J^3(A)\sim \Phi^3(A)\sim A\sim M$, which is exactly the desired conclusion of Theorem~\ref{main1}.

The advantages of this reformulation are pretty obvious. First, we have reduced the number of free parameters from $9$ to $4$ with 5 entries in the matrices from $\widehat{M}_3(R)$ being equal to $1$. Next, there is no need to compute the third power of a rather complicated map, which turned out to be nearly impossible (at least the result of this computation, which took about 20 pages is difficult to use). Finally, we do not have to deal with arbitrary pairs of invertible diagonal matrices: the equivalence relation $\sim$ on $\widehat M_3(R)$ is reduced to the conjugacy. We shall prove Theorem~\ref{main1} by means of proving Theorems~\ref{main2} and~\ref{DOMA}.

The proof consists of deriving first of all the closed expression for the transformation $\Phi$ and $\Phi^{-1}$. Then we derive explicit expressions for $\Phi^2$, and find a conjugating element between
$\Phi^2$ and $\Phi^{-1}$.

\section{Proof of the main results}

Throughout this section $n=3$. The main objective of this section is to prove Theorem~\ref{DOMA} and Theorem~\ref{main2}.

\subsection{Closed formulas for $\Phi$ and $\Phi^{-1}$ and domains}

\begin{lemma}\label{inve} Let
$$
A=\left(\begin{array}{ccc}
1&1&1\\
1&a&b\\
1&c&d
\end{array}\right)\in{\widehat M}_3(R)
\quad \text{and}\quad
B=\left(\begin{array}{cc}
a-1&b-1\\
c-1&d-1
\end{array}\right)\in M_2(R).
$$
Then $A$ is invertible if and only if $B$ is invertible. Moreover, if $B$ is invertible and
$$
B^{-1}=\left(\begin{array}{cc}
s&t\\
u&v
\end{array}\right),
$$
then the inverse of $A$ is given by the formula
\begin{equation}\label{ian}
A^{-1}=\left(\begin{array}{ccc}
1+s+t+u+v&-s-u&-t-v\\
-s-t&s&t\\
-u-v&u&v
\end{array}\right).
\end{equation}
Furthermore, if $a-1$, $b-1$, $c-1$ and $d-1$ are invertible, then $s$, $t$, $u$ and $v$ are given by the following formula $($thus providing an explicit expression of $A^{-1}$ in terms of $a$, $b$, $c$ and $d)$:
\begin{equation}\label{stuv1}
\begin{array}{ll}
s=((a-1)-(b-1)(d-1)^{-1}(c-1))^{-1},&\ \ \ t=((c-1)-(d-1)(b-1)^{-1}(a-1))^{-1},
\\
u=((b-1)-(a-1)(c-1)^{-1}(d-1))^{-1},&\ \ \ v=((d-1)-(c-1)(a-1)^{-1}(b-1))^{-1}.
\end{array}
\end{equation}
\end{lemma}

\begin{proof} Let $s,t,u,v\in R$, $C\in M_3(R)$ be the matrix in the right-hand side of (\ref{ian}) and $D\in M_2(R)$ be the lower-left $2\times 2$ corner of $C$. It is straightforward to verify that $BD=DB=I\iff AC=CA=I$. In particular, the invertibility of $B$ implies the invertibility of $A$ and the equality $B^{-1}=D$ implies $A^{-1}=C$. On the other hand, if $A$ is invertible and $D$ is the left-lower $2\times 2$ corner of $A^{-1}$, it is easy to see that the equation $AA^{-1}=A^{-1}A=I$ implies that $BD=DB=I$ and therefore $B$ is invertible with $B^{-1}=D$. Thus $A$ is invertible if and only if $B$ is invertible and (\ref{ian}) holds provided $B^{-1}=D$. Finally, if $a-1$, $b-1$, $c-1$ and $d-1$ are invertible, Lemma~\ref{2t2inv} applied to the matrix $B$ gives the explicit formulas (\ref{stuv1}) for the entries $s,t,u,v$ of $B^{-1}$.
\end{proof}

We will need few following specific noncommutative identities.

\begin{lemma}\label{iden2} Let $a$, $b$, $c$, $d$ be invertible elements of $R$. Then
\begin{itemize}
\item[\rm (a)]$a-b$ is invertible if and only if $a^{-1}-b^{-1}$ is invertible$;$
\item[\rm (b)]if $a-b$ is invertible, then $a(a-b)^{-1}b=b(a-b)^{-1}a=(b^{-1}-a^{-1})^{-1};$
\item[\rm (c)]if $d-c$ and $b-a$ are invertible, then $db^{-1}-ca^{-1}$ is invertible if and only if
$(d-c)^{-1}c-(b-a)^{-1}a$ is invertible$;$
\item[\rm (d)]if $d-c$, $b-a$ $a-1$, $b-1$, $c-1$ and $d-1$ are invertible, then $(d-1)(b-1)^{-1}-(c-1)(a-1)^{-1}$ is invertible if and only if $(d-c)^{-1}(c-1)-(b-a)^{-1}(a-1)$ is invertible$;$
\item[\rm (e)]if $d-c$, $b-a$ $a-1$, $b-1$, $c-1$ and $d-1$ are invertible, then $(d^{-1}-1)(c^{-1}-1)^{-1}-(b^{-1}-1)(a^{-1}-1)^{-1}$ is invertible if and only if $(c-1)(d-c)^{-1}d-(a-1)(b-a)^{-1}b$ is invertible.
\end{itemize}
\end{lemma}

\begin{proof} (b) is trivial and trivially implies (a). Assume that $d-c$ and $b-a$ are invertible. Multiplying $db^{-1}-ca^{-1}$ by $b$ on the right and by $c^{-1}$ on the left, we see that
$$
db^{-1}-ca^{-1}\in R^*\iff c^{-1}d-a^{-1}b=(c^{-1}d-1)-(a^{-1}b-1)\in R^*.
$$
Using (a), we get
$$
db^{-1}-ca^{-1}\in R^*\iff(c^{-1}d-1)^{-1}-(a^{-1}b-1)^{-1}=(d-c)^{-1}c-(b-a)^{-1}a\in R^*,
$$
which proves (c). Clearly, (d) is (c) applied to $a-1$, $b-1$, $c-1$ and $d-1$ instead of $a$, $b$, $c$ and $d$. Next,
$$
(d^{-1}-1)(c^{-1}-1)^{-1}-1=(d^{-1}-1)(c^{-1}-1)^{-1}-(c^{-1}-1)(c^{-1}-1)^{-1}=(d^{-1}-c^{-1})(c^{-1}-1)^{-1}.
$$
Hence
$$
u=(d^{-1}-1)(c^{-1}-1)^{-1}-(b^{-1}-1)(a^{-1}-1)^{-1}=(d^{-1}-c^{-1})(c^{-1}-1)^{-1}-(b^{-1}-a^{-1})(a^{-1}-1)^{-1}.
$$
Using (a), we see that the invertibility of $u$ is equivalent to the invertibility of
$$
(c^{-1}-1)(d^{-1}-c^{-1})^{-1}-(a^{-1}-1)(b^{-1}-a^{-1})^{-1}=-(c-1)c^{-1}(d^{-1}-c^{-1})^{-1}d^{-1}d+
(a-1)a^{-1}(b^{-1}-a^{-1})^{-1}b^{-1}b.
$$
By (b), $a^{-1}(b^{-1}-a^{-1})^{-1}b^{-1}=(a-b)^{-1}$ and $c^{-1}(d^{-1}-c^{-1})^{-1}d^{-1}=(c-d)^{-1}$. Plugging this into the above equality, we see that $u$  is invertible if and only if $(c-1)(d-c)^{-1}d-(a-1)(b-a)^{-1}b$ is invertible, which completes the proof of (e).
\end{proof}

\begin{lemma}\label{subma} Let
$$
A= \left(\begin{array}{ccc}
1&1&1\\
1&a&b\\
1&c&d
\end{array}\right)\in \widehat{M}_3(R).
$$
Then all $1\times 1$ and $2\times 2$ submatrices of $A$ are invertible if and only if $a$, $b$, $c$, $d$, $a-1$, $b-1$, $c-1$, $d-1$, $d-c$, $d-b$, $c-a$, $b-a$ and $db^{-1}-ca^{-1}$ are invertible in $R$.

Furthermore, if all $1\times 1$ and $2\times 2$ submatrices of $A$ are invertible, then
\begin{align}
A\in M^*_3(R)&\iff (d-c)^{-1}(c-1)-(b-a)^{-1}(a-1)\in R^*;\label{AINV}
\\
J_2(A)\in M^*_3(R)&\iff (c-1)(d-c)^{-1}d-(a-1)(b-a)^{-1}b\in R^*. \label{J2AINV}
\end{align}
In particular,
\begin{equation}\label{SSS}
\widehat{{\cal S}}=\left\{A=\left(\begin{array}{ccc}
1&1&1\\
1&a&b\\
1&c&d
\end{array}\right)\in{\widehat M}_3(R):
\begin{array}{l}
\text{the $15$ elements $a$, $b$, $c$, $d$, $a-1$, $b-1$, $c-1$, $d-1$,}\\
\text{$d-c$, $d-b$, $c-a$, $b-a$, $db^{-1}-ca^{-1}$,}\\
\text{$(d-c)^{-1}(c-1)-(b-a)^{-1}(a-1)$ and}\\
\text{$(c-1)(d-c)^{-1}d-(a-1)(b-a)^{-1}a$ are invertible in $R$.}
\end{array}\right\},
\end{equation}
where $\widehat{{\cal S}}={\cal S}\cap\widehat{M}_3(R)$ as defined in $(\ref{SS})$.
\end{lemma}

\begin{proof} Obviously, the $1\times 1$ submatrices of $A$ are invertible if and only if $a$, $b$, $c$ and $d$ are invertible, in which case $A\in{\rm dom}(J_2)$. Using Lemma~\ref{2t2inv} and Lemma~\ref{iden2}, we see that the invertibility of the $2\times 2$ submatrices
\begin{align*}
&\left(\begin{array}{cc}
1&1\\
1&a
\end{array}\right),\ \ \ \left(\begin{array}{cc}
1&1\\
1&b
\end{array}\right),\ \
\left(\begin{array}{cc}
1&1\\
1&c
\end{array}\right)\ \ \text{and}\ \
\left(\begin{array}{cc}
1&1\\
1&d
\end{array}\right),\ \
\left(\begin{array}{cc}
1&a\\
1&c
\end{array}\right),
\\
&\left(\begin{array}{cc}
1&b\\
1&d
\end{array}\right),\ \
\left(\begin{array}{cc}
1&1\\
a&b
\end{array}\right),\ \
\left(\begin{array}{cc}
1&1\\
c&d
\end{array}\right)\ \ \text{and}\ \ \left(\begin{array}{cc}
a&b\\
c&d
\end{array}\right)
\end{align*}
of $A$ is equivalent to the invertibility of $a-1$, $b-1$, $c-1$, $d-1$, $c-a$, $d-b$, $b-a$, $d-c$ and $db^{-1}-ca^{-1}$
respectively. Summarizing the above observations, we get that all $1\times 1$ and $2\times 2$ submatrices of  $A$ are invertible precisely when the 13 elements $a$, $b$, $c$, $d$, $a-1$, $b-1$, $c-1$, $d-1$, $d-c$, $d-b$, $c-a$, $b-a$ and $db^{-1}=ca^{-1}$ are invertible in $R$.

Now we shall assume that these 13 elements are invertible. By Lemma~\ref{inve}, $A$ is invertible if and only if
$$
B=\left(\begin{array}{cc}
a-1&b-1\\
c-1&d-1\end{array}\right)
$$
is invertible. Since $B\in M_2^\star(R)$, by Lemma~\ref{2t2inv}, this happens if and only if $(d-1)(b-1)^{-1}-(c-1)(a-1)^{-1}$ is invertible. By Part~(d) of Lemma~\ref{iden2}, the last invertibility is equivalent to the invertibility of $(d-c)^{-1}(c-1)-(b-a)^{-1}(a-1)$. This completes the proof of (\ref{AINV}).

Next, by Lemma~\ref{inve}, $J_2(A)$ is invertible if and only if
$$
C=\left(\begin{array}{cc}
a^{-1}-1&c^{-1}-1\\
b^{-1}-1&d^{-1}-1\end{array}\right)
$$
is invertible. Since $C\in M_2^\star(R)$, Lemma~\ref{2t2inv} implies that the latter happens if and only if
$(d^{-1}-1)(c^{-1}-1)^{-1}-(b^{-1}-1)(a^{-1}-1)^{-1}$ is invertible. By Part~(e) of Lemma~\ref{iden2}, this invertibility is equivalent to the invertibility of $(c-1)(d-c)^{-1}d-(a-1)(b-a)^{-1}b$, which completes the proof of (\ref{J2AINV}). Finally, the above observations imply (\ref{SSS}).
\end{proof}

\begin{lemma}\label{meaow} For $M\in M_3(R)$,
\begin{align}
&M\in {\cal S}\ \Longrightarrow\ J_2(M)\in{\cal S}; \label{J2S}
\\
&M\in {\cal S}\ \Longrightarrow\ \text{all square submatrices of $M^{-1}$ are invertible}. \label{J1S}
\end{align}
Furthermore,
\begin{equation}\label{domdom}
{\rm dom}(J)\cap {\rm dom}(J^{-1})={\cal S}\ \ \text{and}\ \ {\rm dom}(\Phi)\cap {\rm dom}(\Phi^{-1})=\widehat{\cal S}.
\end{equation}
\end{lemma}

\begin{proof} Let $M\in {\cal S}$. Then all $1\times 1$ submatrices of $J_2(M)$ as well as $J_2(M)$ and $M=J_2(J_2(M))$ are invertible. In order to show that $J_2(M)\in{\cal S}$, it suffices to verify that all $2\times 2$ submatrices of $J_2(M)$ are invertible. Since the latter are the $J_2$-images of $2\times 2$ submatrices of $M$, which are invertible, Lemma~\ref{2t2inv} implies that $2\times 2$ submatrices of $J_2(M)$ are invertible as well. Thus $J_2(M)\in{\cal S}$, which proves (\ref{J2S}).

Next, we  verify (\ref{J1S}). The invertibility of $M^{-1}$ itself is not an issue. Thus it suffices to verify the invertibility of the $1\times 1$ and $2\times 2$ submatrices of $M^{-1}$. Recall that the action of the group ${\cal P}_{LR}$ of pairs of permutation matrices leaves $\cal S$, $M_3^*(R)$ and $M_3^\star(R)$ invariant and both $J_1$ and $J_2$ act on the orbits of this action. Furthermore, ${\cal P}_{LR}$-action permutes transitively the $2\times 2$ submatrices as well as the $1\times 1$ submatrices of a given matrix. This observation implies that it is enough to check the invertibility of specific $1\times 1$ and $2\times 2$ submatrices of $M^{-1}$. We choose the lower right corner $1\times 1$ and $2\times 2$ submatrices of $M^{-1}$. Next, it is easy to see that ${\cal D}_{LR}$-action preserves $\cal S$ and does not disturb the invertibility of any given square submatrix. Thus nothing changes if we replace $M$ by any matrix in the same ${\cal D}_{LR}$-orbit. By Lemma~\ref{4pA}, we can find a matrix from ${\widehat M}_3(R)$ in the ${\cal D}_{LR}$-orbit of $M$. Thus, without loss of generality, we may assume that
$$
M=\left(\begin{array}{ccc}
1&1&1\\
1&a&b\\
1&c&d
\end{array}\right)\in \widehat{M}_3(R).
$$
Since $M\in{\cal S}$, we have $M\in\widehat{\cal S}$ and Lemma~\ref{subma} implies that the 15 elements from (\ref{SSS}) are invertible. By Lemma~\ref{inve},
$$
M^{-1}=\left(\begin{array}{ccc}
1+s+t+u+v&-s-u&-t-v\\
-s-t&s&t\\
-u-v&u&v
\end{array}\right),
$$
where $s$, $t$, $u$ and $v$ are given by (\ref{stuv1}) and satisfy
$$
\left(\begin{array}{cc}
a-1&b-1\\
c-1&d-1
\end{array}\right) \left(\begin{array}{cc}
s&t\\
u&v
\end{array}\right)=\left(\begin{array}{cc}
s&t\\
u&v
\end{array}\right)\left(\begin{array}{cc}
a-1&b-1\\
c-1&d-1
\end{array}\right)=\left(\begin{array}{cc}
1&0\\
0&1
\end{array}\right).
$$
This immediately implies that the lower right corner $2\times 2$ submatrix of $M^{-1}$ is invertible. The invertibility of the lower right corner $1\times 1$ submatrix of $M^{-1}$ is the invertibility of $v$, which follows from (\ref{stuv1}).
Thus the square submatrices of $M^{-1}$ are invertible. This completes the proof of (\ref{J1S}).

Next, we will check that $M\in {\rm dom}(J)\cap {\rm dom}(J^{-1})$ (we still assume that $M\in{\cal S}$).
Since $1\times 1$-submatices of $M$ are invertible, $M\in {\rm dom}(J_2)$. Since $J_2(M)$ is invertible, $M\in{\rm dom}(J^{-1})$. Since $M$ is invertible, $M\in{\rm dom}(J_1)$. By the already proved (\ref{J1S}), $J_1(M)$ is Hadamard invertible. Hence $M\in{\rm dom}(J)$. This proves the inclusion ${\rm dom}(J)\cap {\rm dom}(J^{-1})\supseteq {\cal S}$.

Now assume that $A\in\widehat{\cal S}$. Then $A\in \widehat{M}_3(R)\cap M^\star_3(R)\cap M^*_3(R)$. By the already verified (\ref{J1S}), $J_1(A)\in M^\star_3(R)$. Thus $A\in {\rm dom}(\Phi)$. By the already verified (\ref{J2S}), $J_2(A)\in \widehat{\cal S}$. Hence $J_2(A)$ is invertible and $J_2(A)^{-1}$ is Hadamard invertible. It follows that $A\in{\rm dom}(\Phi^{-1})$. This proves the inclusion ${\rm dom}(\Phi)\cap {\rm dom}(\Phi^{-1})\supseteq \widehat{\cal S}$.

Next, assume that $M\in {\rm dom}(J)\cap {\rm dom}(J^{-1})$.
This means that $M$ is both invertible and Hadamard invertible, that $M^{-1}$ is Hadamard invertible and $J_2(M)$ is invertible. In particular, $M$ and $J_2(M)$ as well as all $1\times 1$ submatrices of $M$ are invertible. In order to complete the proof of the inclusion $M\in{\cal S}$, it remains to show that all $2\times 2$ submatrices of $M$ are invertible. As in the first part of the proof, using the actions of ${\cal P}_{LR}$ and ${\cal D}_{LR}$, we can reduce the task to the verification of the invertibility of the right upper corner $2\times 2$ submatrix of $M$ in the case $M\in {\widehat M}_3(R)$. In this case, by Lemma~\ref{inve},
\begin{align*}
&M=\left(\begin{array}{ccc}
1&1&1\\
1&a&b\\
1&c&d
\end{array}\right)\ \ \ \text{and}\ \ \ M^{-1}=\left(\begin{array}{ccc}
1+s+t+u+v&-s-u&-t-v\\
-s-t&s&t\\
-u-v&u&v
\end{array}\right),
\\
&\qquad\text{where}\ \ \ B=\left(\begin{array}{cc}
a-1&b-1\\
c-1&d-1
\end{array}\right)\ \ \text{is invertible and}\ \ B^{-1}=\left(\begin{array}{cc}
s&t\\
u&v
\end{array}\right).
\end{align*}
Since $s$, $t$, $u$ and $v$ are entries of $M^{-1}$, which happens to be Hadamard invertible, $B^{-1}$ is Hadamard invertible. By Lemma~\ref{2t2inv}, $B$ is Hadamard invertible. That is, $a-1$, $b-1$, $c-1$ and $d-1$ are invertible.
By Lemma~\ref{2t2inv}, the invertibility of $a-1$ yields the invertibility of
the right upper corner $2\times 2$ submatrix of $M$. Thus $M\in{\cal S}$. This completes the proof of the inclusion ${\rm dom}(J)\cap {\rm dom}(J^{-1})\subseteq {\cal S}$.

Since ${\rm dom}(\Phi)\subseteq {\rm dom}(J)\cap \widehat{M}_3(R)$, ${\rm dom}(\Phi^{-1})\subseteq {\rm dom}(J^{-1})\cap \widehat{M}_3(R)$ and ${\rm dom}(J)\cap {\rm dom}(J^{-1})\subseteq {\cal S}$, we have ${\rm dom}(\Phi)\cap {\rm dom}(\Phi^{-1})\subseteq {\cal S}\cap \widehat{M}_3(R)=\widehat{\cal S}$. This completes the proof of (\ref{domdom}).
\end{proof}

By Lemma~\ref{meaow}, $\Phi(A)$ and $\Phi^{-1}(A)$ are well-defined if $A\in\widehat{\cal S}$. Together with Corollary~\ref{kley1}, the following lemma provides a explicit formulas for $\Phi(A)$ and $\Phi^{-1}(A)$ provided $A\in\widehat{\cal S}$.

\begin{lemma}\label{nneeww} (Closed formulas for  $\Phi(A)$ and $\Phi^{-1}(A)$)
Let
\begin{equation}\label{notat}
A= \left(\begin{array}{ccc}
1&1&1\\
1&a&b\\
1&c&d
\end{array}\right)\in \widehat{\cal S},\ \ \
\Phi(A)=\left(\begin{array}{ccc}
1&1&1\\
1&a'&b'\\
1&c'&d'
\end{array}\right),\ \ \text{and}\ \ \Phi^{-1}(A)=\left(\begin{array}{ccc}
1&1&1\\
1&a^\circ&b^\circ\\
1&c^\circ&d^\circ
\end{array}\right).
\end{equation}
Then
\begin{equation}\label{hfex}
\begin{array}{l}
a'=(d-1)^{-1}(d-c)a^{-1}(db^{-1}-ca^{-1})^{-1}(db^{-1}-1);
\\
a^\circ=(d-1)(d-c)^{-1}(db^{-1}-ca^{-1})(db^{-1}-1)^{-1}.
\end{array}
\end{equation}
\end{lemma}

\begin{proof}
It is possible to verify the formulas (\ref{hfex}) by simplifying the expressions provided by the definition of $\Phi$. However, an indirect approach proves to be shorter. Let $C=\lal(A^{-1})$.
By definition of $\Phi$,
$$
C=\lal(A^{-1})=J_2(\Phi(A))=\left(\begin{array}{ccc}
1&1&1\\
1&{a'}^{-1}&{c'}^{-1}\\
1&{b'}^{-1}&{d'}^{-1}
\end{array}\right).
$$
By definition of $\lal$, there are unique invertible diagonal $3\times 3$ matrices $D_1,D_2\in M_3(R)$ such that the top left entry of $D_1$ is $1$ and $C=\lal(A^{-1})=D_1^{-1}A^{-1}D_2$. It follows that $AD_1C=D_2$ is diagonal. Let $1,x,y$ be the diagonal entries of $D_1$. Then, using the above, we get
$$
AD_1C=\left(\begin{array}{ccc}
1&1&1\\
1&a&b\\
1&c&d
\end{array}\right)
\left(\begin{array}{ccc}
1&0&0\\
0&x&0\\
0&0&y
\end{array}\right)
\left(\begin{array}{ccc}
1&1&1\\
1&{a'}^{-1}&{c'}^{-1}\\
1&{b'}^{-1}&{d'}^{-1}
\end{array}\right)
\ \ \ \text{is diagonal.}
$$
A direct computation of the off-diagonal entries of $AD_1C$ shows that the above can be rewritten as the following system of six equations:
\begin{equation}\label{sysy}
\begin{array}{l}1+ax+by=0;\\ 1+cx+dy=0;\end{array}\quad
\begin{array}{l}1+x{a'}^{-1}+y{b'}^{-1}=0;\\ 1+cx{a'}^{-1}+dy{b'}^{-1}=0;\end{array}\ \
\begin{array}{l} 1+x{c'}^{-1}+y{d'}^{-1}=0;\\ 1+ax{c'}^{-1}+by{d'}^{-1}=0.
\end{array}
\end{equation}
The first pair of equations allows to determine $x$. Multiplying the first equation by $db^{-1}$ on the left and subtracting the result from the second equation, we get $(1-db^{-1})+(c-db^{-1}a)x=0$. This gives $x=a^{-1}(ca^{-1}-db^{-1})^{-1}(db^{-1}-1)$. Now we use the second pair to find $a'$. Namely, multiplying the fourth equation by $d^{-1}$ on the left and subtracting the result from the third equation, we get
$$
(1-d^{-1})+(1-d^{-1}c)x{a'}^{-1}=0\iff a'=-(1-d^{-1})^{-1}(1-d^{-1}c)x=(d-1)^{-1}(c-d)x.
$$
Plugging in $x=a^{-1}(ca^{-1}-db^{-1})^{-1}(db^{-1}-1)$, we obtain
$$
a'=f(a,b,c,d)=(d-1)^{-1}(c-d)a^{-1}(ca^{-1}-db^{-1})^{-1}(db^{-1}-1).
$$

The already used identity $\lal(A^{-1})=J_2(\Phi(A))$ applied to $\Phi(A)$ instead of $A$ gives
$J_2(A)=\lal((\Phi^{-1}(A))^{-1})$. By definition of $\lal$, there are unique invertible diagonal $3\times 3$ matrices $D_3,D_4\in M_3(R)$ such that the top left entry of $D_3$ is $1$ and $J_2(A)=D_3^{-1}\Phi^{-1}(A)^{-1}D_4$. Hence $\Phi^{-1}(A)D_3J_2(A)=D_4$. Let $1,u,v$ be the diagonal entries of $D_3$. Then, the equality $\Phi^{-1}(A)D_3J_2(A)=D_4$ yields
$$
\Phi(A)D_3J_2(A)=\left(\begin{array}{ccc}
1&1&1\\
1&a^\circ&b^\circ\\
1&c^\circ&d^\circ
\end{array}\right)
\left(\begin{array}{ccc}
1&0&0\\
0&u&0\\
0&0&v
\end{array}\right)
\left(\begin{array}{ccc}
1&1&1\\
1&{a}^{-1}&{c}^{-1}\\
1&{b}^{-1}&{d}^{-1}
\end{array}\right)
\ \ \ \text{is diagonal.}
$$
A direct computation of the off-diagonal entries of $\Phi^{-1}(A)D_3J_2(A)$ shows that the above matrix equation can be rewritten as the following system of equations:
\begin{equation}\label{sysy1}
\begin{array}{l}1+ua^{-1}+vb^{-1}=0;\\ 1+uc^{-1}+vd^{-1}=0\end{array}\quad
\begin{array}{l}1+a^\circ u+b^\circ v=0;\\ 1+a^\circ u{c}^{-1}+b^\circ v{d}^{-1}=0;\end{array}\ \
\begin{array}{l}1+c^\circ u+d^\circ v=0;\\ 1+c^\circ u{a}^{-1}+d^\circ v{b}^{-1}=0.\end{array}
\end{equation}
The first pair of equations allows to determine $u$. Multiplying the first equation by $bd^{-1}$ on the right and subtracting the result from the second equation, we get $(1-bd^{-1})+u(c^{-1}-a^{-1}bd^{-1})=0$. This gives $u=-(d-b)(c^{-1}d-a^{-1}b)^{-1}$. Now we use the second pair to find $a^\circ$. Namely, multiplying the fourth equation by $d$ on the right and subtracting the result from the third equation, we get
$$
(1-d)+a^\circ u(1-c^{-1}d)=0\iff a^\circ=-(1-d)(1-c^{-1}d)^{-1}u^{-1}.
$$
Plugging in $u=-(d-b)(c^{-1}d-a^{-1}b)^{-1}$, we obtain
$$
a^\circ=h(a,b,c,d)=(d-1)(d-c)^{-1}(db^{-1}-ca^{-1})(db^{-1}-1)^{-1}.
$$
Note that in the above computation, we have never inverted an element whose invertibility was not guaranteed by (\ref{SSS}). \end{proof}

Certainly, systems (\ref{sysy}) and (\ref{sysy1}) allow to get explicit expressions for all of $a'$, $b'$, $c'$, $d'$, $a^\circ$, $b^\circ$, $c^\circ$ and $d^\circ$.
Due to the Symmetry Lemma (Corollary~\ref{kley1}) we need to do it only for one element.

Sometimes the other way to write $\Phi(A)$ is more convenient.

\begin{lemma}\label{phiano} Let $A\in\widehat{\cal S}$. Keeping the notation introduced in $(\ref{notat})$, we can write the entries of $\Phi(A)$ in the following way:
\begin{equation}\label{phiconv}
\begin{array}{l}
a'=(d-1)^{-1}(d-c)a^{-1}(db^{-1}-ca^{-1})^{-1}(db^{-1}-1);\\
b'=(c-1)^{-1}(d-c)b^{-1}(db^{-1}-ca^{-1})^{-1}(ca^{-1}-1);\\
c'=(b-1)^{-1}(b-a)a^{-1}(db^{-1}-ca^{-1})^{-1}(db^{-1}-1);\\ d'=(a-1)^{-1}(b-a)b^{-1}(db^{-1}-ca^{-1})^{-1}(ca^{-1}-1).\end{array}
\end{equation}
\end{lemma}

\begin{proof} By Lemma~\ref{nneeww} and Corollary~\ref{kley1},
\begin{equation}\label{phiconvN}
\begin{array}{l}
a'=(d-1)^{-1}(d-c)a^{-1}(db^{-1}-ca^{-1})^{-1}(db^{-1}-1);\\
b'=(c-1)^{-1}(c-d)b^{-1}(ca^{-1}-db^{-1})^{-1}(ca^{-1}-1);\\
c'=(b-1)^{-1}(b-a)c^{-1}(bd^{-1}-ac^{-1})^{-1}(bd^{-1}-1);\\ d'=(a-1)^{-1}(a-b)d^{-1}(ac^{-1}-bd^{-1})^{-1}(ac^{-1}-1).\end{array}
\end{equation}
The first equalities in (\ref{phiconv}) and in (\ref{phiconvN}) are identical. The second equality in (\ref{phiconv}) is equivalent to the second equality in (\ref{phiconvN}): the right-hand sides differ by changing signs of two terms in the product.
We need to essentially rewrite expressions for $c'$ and $d'$ in order to calculate more conveniently $\Phi^2$. For passing from one expression for  $c'$ and $d'$ to another it is enough to
verify the identities
$$
\begin{array}{l}
(b-1)^{-1}(b-a)a^{-1}(db^{-1}-ca^{-1})^{-1}(db^{-1}-1)=(b-1)^{-1}(b-a)c^{-1}(bd^{-1}-ac^{-1})^{-1}(bd^{-1}-1);\\
(a-1)^{-1}(b-a)b^{-1}(db^{-1}-ca^{-1})^{-1}(ca^{-1}-1)=(a-1)^{-1}(a-b)d^{-1}(ac^{-1}-bd^{-1})^{-1}(ac^{-1}-1).
\end{array}
$$
Since $(b-1)^{-1}$ and $b-a$ are invertible (see (\ref{SSS})), the equalities in the above display are equivalent to
$$
\begin{array}{l}
a^{-1}(db^{-1}-ca^{-1})^{-1}(db^{-1}-1)=c^{-1}(bd^{-1}-ac^{-1})^{-1}(bd^{-1}-1);\\
b^{-1}(db^{-1}-ca^{-1})^{-1}(ca^{-1}-1)=-d^{-1}(ac^{-1}-bd^{-1})^{-1}(ac^{-1}-1).
\end{array}
$$
Since $bd^{-1}-1=-bd^{-1}(db^{-1}-1)$ and $ac^{-1}-1=-ac^{-1}(ca^{-1}-1)$, we can rewrite the above display as
$$
\begin{array}{l}
a^{-1}(db^{-1}-ca^{-1})^{-1}(db^{-1}-1)=-c^{-1}(bd^{-1}-ac^{-1})^{-1}bd^{-1}(db^{-1}-1);\\
b^{-1}(db^{-1}-ca^{-1})^{-1}(ca^{-1}-1)=d^{-1}(ac^{-1}-bd^{-1})^{-1}ac^{-1}(ca^{-1}-1).
\end{array}
$$
From (\ref{SSS}) it follows that $db^{-1}-1$ and $ca^{-1}-1$ are invertible. This allows us to multiply the above equations by $(db^{-1}-1)^{-1}$ and $(ca^{-1}-1)^{-1}$ on the right respectively. While we are at it, we also multiply these equations on the left by the invertible elements $c$ and $d$ respectively. The equations  take form
$$
ca^{-1}(db^{-1}-ca^{-1})^{-1}=-(bd^{-1}-ac^{-1})^{-1}bd^{-1},\ \ db^{-1}(db^{-1}-ca^{-1})^{-1}=(ac^{-1}-bd^{-1})^{-1}ac^{-1}.
$$
Pulling the out-of-bracket terms inside, we see that these equations are equivalent to
$$
(db^{-1}ac^{-1}-1)^{-1}=-(1-db^{-1}ac^{-1})^{-1},\ \ (1-ca^{-1}bd^{-1})^{-1}=(1-ca^{-1}bd^{-1})^{-1}.
$$
Both equations are trivially satisfied, which completes the proof.
\end{proof}

Now we will prove the key fact for our statement about domains.

\begin{lemma}\label{invphi} Let $M\in {\cal S}$. Then $M^{-1}\in {\cal S}$.
\end{lemma}

\begin{proof} By Lemma~\ref{meaow}, all square submatrices of $M$ are invertible. Thus it suffices to prove that $J_2(M^{-1})=J(M)$ is invertible. By Lemma~\ref{4pA}, we can choose
$$
A= \left(\begin{array}{ccc}
1&1&1\\
1&a&b\\
1&c&d
\end{array}\right)\in {\widehat M}_3(R)\ \ \text{such that $A\sim M$.}
$$
Then $A\in \widehat{\cal S}\subseteq{\rm dom}(\Phi)$. By (\ref{orb1}) and (\ref{phij}), $J(M)\sim J(A)\sim \Phi(A)$. Hence the proof will be complete if we demonstrate that $\Phi(A)$ is invertible.

Since $1\times 1$ and $2\times 2$ submatrices of $M^{-1}$ are invertible, Lemma~\ref{2t2inv} guarantees that the same happens for $J_2(M^{-1})=J(M)\sim\Phi(A)$. Thus the $1\times 1$ and $2\times 2$ submatrices of $\Phi(A)$ are invertible. By Lemma~\ref{phiano},
$$
\Phi(A)=\left(\begin{array}{ccc}
1&1&1\\
1&a'&b'\\
1&c'&d'
\end{array}\right)\ \ \text{with $a'$, $b'$, $c'$ and $d'$ given by (\ref{phiconv}).}
$$
Since $1\times 1$ and $2\times 2$ submatrices of $\Phi(A)$ are invertible, Lemma~\ref{subma} implies that the 13 elements
$a'$, $b'$, $c'$, $d'$, $a'-1$, $b'-1$, $c'-1$, $d'-1$, $d'-c'$, $d'-b'$, $c'-a'$, $b'-a'$ and $d'{b'}^{-1}-c'{a'}^{-1}$ are invertible in $R$. The same lemma guarantees that the invertibility of $\Phi(A)$ is equivalent to the invertibility of
$$
q=(d'-c')^{-1}(c'-1)-(b'-a')^{-1}(a'-1).
$$
By (\ref{phiconv}),
\begin{equation}\label{cd}
\begin{array}{l}
c'=(b-1)^{-1}(b-a)a^{-1}(db^{-1}-ca^{-1})^{-1}(db^{-1}-1);\\ d'=(a-1)^{-1}(b-a)b^{-1}(db^{-1}-ca^{-1})^{-1}(ca^{-1}-1).\end{array}
\end{equation}
Denote
$$
x=(db^{-1}-ca^{-1})^{-1}(ca^{-1}-1)\ \ \text{and}\ \ y=a(b-a)^{-1}(b-1).
$$
Next, observe that
\begin{align*}
x+1&=(db^{-1}-ca^{-1})^{-1}(ca^{-1}-1)+(db^{-1}-ca^{-1})^{-1}(db^{-1}-ca^{-1})
\\
&=(db^{-1}-ca^{-1})^{-1}[ca^{-1}-1+db^{-1}-ca^{-1}]=(db^{-1}-ca^{-1})^{-1}(db^{-1}-1)\ \ \ \text{and}
\\
y-1&=a(b-a)^{-1}b-a(b-a)^{-1}-(b-a)(b-a)^{-1}=a(b-a)^{-1}b-(b-a+a)(b-a)^{-1}
\\
&=a(b-a)^{-1}b-b(b-a)^{-1}=b(b-a)^{-1}a-b(b-a)^{-1}=b(b-a)^{-1}(a-1),
\end{align*}
where the second last equality follows from Part (b) of Lemma~\ref{iden2}. Plugging the data from the above two expressions into (\ref{cd}), we get
$$
d'=(y-1)^{-1}x\ \ \ \text{and}\ \ \ c'=y^{-1}(x+1).
$$
Thus
\begin{align*}
(d'-c')^{-1}(c'-1)&=((y-1)^{-1}x-y^{-1}(x+1))^{-1}(y^{-1}(x+1)-1)
\\
&=(yx-(y-1)(x+1))^{-1}(y-1)yy^{-1}(x+1-y)=(x+1-y)^{-1}(y-1)(x+1-y).
\end{align*}
Denote
$$
w=b^{-1}(b-a)a^{-1}(x+1-y).
$$
Then $w$ is invertible and the above two formulas yield
$$
(d'-c')^{-1}(c'-1)=w^{-1}b^{-1}(b-a)a^{-1}(y-1)a(b-a)^{-1}bw.
$$
Using the equality $y-1=b(b-a)^{-1}(a-1)$, we get
\begin{align*}
&b^{-1}(b-a)a^{-1}(y-1)a(b-a)^{-1}b=b^{-1}(b-a)a^{-1}b(b-a)^{-1}(a-1)a(b-a)^{-1}b
\\
&\ \ =a^{-1}(b-a)b^{-1}b(b-a)^{-1}(a-1)a(b-a)^{-1}b=a^{-1}(a-1)a(b-a)^{-1}b=(a-1)(b-a)^{-1}b.
\end{align*}
Next, by the definition of $w$ and the formulas $x+1=(db^{-1}-ca^{-1})^{-1}(db^{-1}-1)$ and $y=a(b-a)^{-1}(b-1)$ we get
\begin{align*}
w&=b^{-1}(b-a)a^{-1}(x+1-y)=b^{-1}(b-a)a^{-1}[(db^{-1}-ca^{-1})^{-1}(db^{-1}-1)-a(b-a)^{-1}(b-1)]
\\
&=b^{-1}(b-a)a^{-1}(db^{-1}-ca^{-1})^{-1}(db^{-1}-1)-b^{-1}(b-a)a^{-1}a(b-a)^{-1}(b-1)
\\
&=b^{-1}(b-a)a^{-1}(db^{-1}-ca^{-1})^{-1}(db^{-1}-1)-b^{-1}(b-1)
\\
&=b^{-1}(b-a)a^{-1}(db^{-1}-ca^{-1})^{-1}(db^{-1}-1)+b^{-1}-1.
\end{align*}
By the above three displays,
\begin{equation}\label{dcre}
\begin{array}{l}
(d'-c')^{-1}(c'-1)=w^{-1}(a-1)(b-a)^{-1}bw,\ \ \text{where}
\\
w=b^{-1}(b-a)a^{-1}(db^{-1}-ca^{-1})^{-1}(db^{-1}-1)+b^{-1}-1.
\end{array}
\end{equation}

By Corollary~\ref{kley1}, an expression of $(b'-a')^{-1}(a'-1)$ in terms of $a$, $b$, $c$ and $d$ can be obtained from the expression (\ref{dcre}) for $(d'-c')^{-1}(c'-1)$ by the letter permutation $(a,b,c,d)\mapsto (c,d,a,b)$. This yields
\begin{equation}\label{bare}
\begin{array}{l}
(b'-a')^{-1}(a'-1)=u^{-1}(c-1)(d-c)^{-1}du,\ \ \text{where}
\\
u=d^{-1}(d-c)c^{-1}(bd^{-1}-ac^{-1})^{-1}(bd^{-1}-1)+d^{-1}-1.
\end{array}
\end{equation}

Now a small miracle happens (frankly speaking, we have replaced the conjugating element $1+x-y$ by $w$ precisely for it to happen). Namely,
\begin{equation*}
u=w.
\end{equation*}
Indeed, the identity $u=w$ is equivalent to
$$
b^{-1}(b-a)a^{-1}(db^{-1}-ca^{-1})^{-1}(db^{-1}-1)+b^{-1}=d^{-1}(d-c)c^{-1}(bd^{-1}-ac^{-1})^{-1}(bd^{-1}-1)+d^{-1}.
$$
Using the obvious equality $(bd^{-1}-ac^{-1})^{-1}=-ca^{-1}(db^{-1}-ca^{-1})^{-1}db^{-1}$, we can replace $(bd^{-1}-ac^{-1})^{-1}$ in the right-hand side of the above display by $-ca^{-1}(db^{-1}-ca^{-1})^{-1}db^{-1}$, which after a couple of cancelations results in the following equality equivalent to $u=w$:
$$
b^{-1}(b-a)a^{-1}(db^{-1}-ca^{-1})^{-1}(db^{-1}-1)+b^{-1}=d^{-1}(d-c)a^{-1}(db^{-1}-ca^{-1})^{-1}(db^{-1}-1)+d^{-1}.
$$
Multiplying both sides by $(db^{-1}-1)^{-1}$ on the right leads to
$$
b^{-1}(b-a)a^{-1}(db^{-1}-ca^{-1})^{-1}+b^{-1}(db^{-1}-1)^{-1}=d^{-1}(d-c)a^{-1}(db^{-1}-ca^{-1})^{-1}+d^{-1}(db^{-1}-1)^{-1}.
$$
Transferring the first summand in the right-hand side to the left and the second summand in the left-hand side to the right we arrive to the following equality still equivalent to $u=w$:
$$
(b^{-1}(b-a)-d^{-1}(d-c))a^{-1}(db^{-1}-ca^{-1})^{-1}=(d^{-1}-b^{-1})(db^{-1}-1)^{-1}.
$$
Since $(db^{-1}-1)^{-1}=-(d^{-1}-b^{-1})^{-1}d^{-1}$, the equality simplifies further:
$$
(b^{-1}(b-a)-d^{-1}(d-c))a^{-1}(db^{-1}-ca^{-1})^{-1}=-d^{-1}.
$$
After multiplying by $(db^{-1}-ca^{-1})a$ on the right, it has a shape
$$
b^{-1}(b-a)-d^{-1}(d-c)=-d^{-1}(db^{-1}-ca^{-1})a.
$$
After opening up the brackets the equality takes form
$$
1-b^{-1}a-1+d^{-1}c=-b^{-1}a+d^{-1}c.
$$
It is obviously correct, so we have an equality $u=w$. Now (\ref{dcre}) and (\ref{bare}) yield
$$
q=(d'-c')^{-1}(c'-1)-(b'-a')^{-1}(a'-1)=-w^{-1}[(c-1)(d-c)^{-1}d-(a-1)(b-a)^{-1}b]w.
$$
Thus the invertibility of $q$ is equivalent to the invertibility of $p=(c-1)(d-c)^{-1}d-(a-1)(b-a)^{-1}b$. The latter is invertible according to Lemma~\ref{subma}. In fact, (\ref{J2AINV}) says that the invertibility of $p$ is equivalent to the invertibility of $J_2(A)$. Since the invertibility of $q$ implies the invertibility of $\Phi(A)$, $\Phi(A)$ is invertible and therefore so is $J_2(M^{-1})$. The proof is complete.
\end{proof}

\subsection{Proof of Theorem~\ref{DOMA}}

By Lemma~\ref{meaow}, ${\rm dom}(J)\cap {\rm dom}(J^{-1})={\cal S}$ and ${\rm dom}(\Phi)\cap {\rm dom}(\Phi^{-1})=\widehat{\cal S}$. If $M\in {\cal S}$, then $J_2(M)\in{\cal S}$ by Lemma~\ref{meaow} and $J_1(M)\in{\cal S}$ by Lemma~\ref{invphi}. Since $J_1$ and $J_2$ are involutions, each of them maps $\cal S$ onto itself bijectively. Hence, so does $J=J_2\circ J_1$.

Next, let $A\in\widehat{\cal S}$. By \ref{phij}, $J(A)\sim \Phi(A)$. Since $J(A)\in{\cal S}$, we have $\Phi(A)\in{\cal S}$. It follows that $\Phi(\widehat{\cal S})\subseteq{\widehat{\cal S}}$. Similar argument gives $\Phi^{-1}(\widehat{\cal S})\subseteq{\widehat{\cal S}}$. Hence $\Phi$ maps $\widehat{\cal S}$ onto itself bijectively.

Now let $k\geq 2$. Since $J$ maps $\cal S$ onto itself bijectively, ${\cal S}\subseteq {\rm dom}(J^k)$. On the other hand, if $M\in{\rm dom}(J^k)$, then $J(M)\in{\rm dom}(J^{-1})$ and $J(M)\in {\rm dom}(J^{k-1})\subseteq {\rm dom}(J)$. Thus $J(M)\in {\rm dom}(J^{-1})\cap {\rm dom}(J)={\cal S}$. Then $M=J^{-1}(J(M))\in {\cal S}$ ($J$ maps $\cal S$ onto itself bijectively). Hence ${\cal S}\supseteq {\rm dom}(J^k)$ and therefore ${\cal S}={\rm dom}(J^k)$. The equality $\widehat{\cal S}={\rm dom}(\Phi^k)$ is proved in exactly the same way.

\subsection{A closed formula for $\Phi^2$}

Now applying previously obtained expression for $\Phi$ we get the following formula for $\Phi^2$.

\begin{lemma}\label{PHI2E} Let
\begin{equation*}
A= \left(\begin{array}{ccc}
1&1&1\\
1&a&b\\
1&c&d
\end{array}\right)\in \widehat{\cal S},\ \ \
\Phi(A)=\left(\begin{array}{ccc}
1&1&1\\
1&a'&b'\\
1&c'&d'
\end{array}\right),\ \ \text{and}\ \ \Phi^2(A)=\left(\begin{array}{ccc}
1&1&1\\
1&a''&b''\\
1&c''&d''
\end{array}\right).
\end{equation*}
Then
\begin{equation}\label{aprpr}
a''=\zeta^{-1}(d-c)^{-1}(db^{-1}-ca^{-1})(db^{-1}-1)^{-1}(d-1)\zeta,\ \ \text{where}\ \
\zeta=[(d-c)^{-1}(c-1)-(b-a)^{-1}(a-1)\bigr].
\end{equation}
 \end{lemma}

 \begin{proof} By Theorem~\ref{DOMA}, $\Phi(A)\in\widehat{S}$. By Lemma~\ref{phiano},
\begin{equation}\label{start}
\begin{array}{l}
a'=(d-1)^{-1}(d-c)a^{-1}(db^{-1}-ca^{-1})^{-1}(db^{-1}-1);\\
b'=(c-1)^{-1}(d-c)b^{-1}(db^{-1}-ca^{-1})^{-1}(ca^{-1}-1);\\
c'=(b-1)^{-1}(b-a)a^{-1}(db^{-1}-ca^{-1})^{-1}(db^{-1}-1);\\
d'=(a-1)^{-1}(b-a)b^{-1}(db^{-1}-ca^{-1})^{-1}(ca^{-1}-1);\\
a''=(d'-1)^{-1}(d'-c'){a'}^{-1}(d'{b'}^{-1}-c'{a'}^{-1})^{-1}(d'{b'}^{-1}-1).\end{array}
\end{equation}
Using these formulas and performing straightforward cancellations, we get
\begin{equation}\label{dbca}
d'{b'}^{-1}=(a-1)^{-1}(b-a)(d-c)^{-1}(c-1)\ \ \text{and}\ \ c'{a'}^{-1}=(b-1)^{-1}(b-a)(d-c)^{-1}(d-1).
\end{equation}
As in the proof of Lemma~\ref{invphi}, denoting $x=(db^{-1}-ca^{-1})^{-1}(ca^{-1}-1)$ and observing that $x+1=(db^{-1}-ca^{-1})^{-1}(db^{-1}-1)$, we can by means of (\ref{start}) write
$$
d'{a'}^{-1}=(a-1)^{-1}(b-a)b^{-1}x(x+1)^{-1}a(d-c)^{-1}(d-1)=(a-1)^{-1}(b-a)b^{-1}(x+1)^{-1}xa(d-c)^{-1}(d-1).
$$
The point of introducing $x$ was to swap $x$ and $(x+1)^{-1}$. Now we plug $x=(db^{-1}-ca^{-1})^{-1}(ca^{-1}-1)$ and $x+1=(db^{-1}-ca^{-1})^{-1}(db^{-1}-1)$ back in:
$$
d'{a'}^{-1}=(a-1)^{-1}(b-a)b^{-1}(db^{-1}-1)^{-1}(db^{-1}-ca^{-1})(db^{-1}-ca^{-1})^{-1}(ca^{-1}-1)a(d-c)^{-1}(d-1).
$$
After obvious cancellations, we get
\begin{equation}\label{da}
d'{a'}^{-1}=(a-1)^{-1}(b-a)(d-b)^{-1}(c-a)(d-c)^{-1}(d-1).
\end{equation}
Plugging (\ref{dbca}), (\ref{da}) and the fourth line of (\ref{start}) into the last line of (\ref{start}), we get
\begin{align*}
a''&=\bigl[(a-1)^{-1}(b-a)b^{-1}(db^{-1}-ca^{-1})^{-1}(ca^{-1}-1)-1\bigr]^{-1}
\\
&\quad\times\bigl[(a-1)^{-1}(b-a)(d-b)^{-1}(c-a)(d-c)^{-1}(d-1)-(b-1)^{-1}(b-a)(d-c)^{-1}(d-1)\bigr]
\\
&\quad\times\bigl[(a-1)^{-1}(b-a)(d-c)^{-1}(c-1)-(b-1)^{-1}(b-a)(d-c)^{-1}(d-1)\bigr]^{-1}
\\
&\quad\times\bigl[(a-1)^{-1}(b-a)(d-c)^{-1}(c-1)-1\bigr].
\end{align*}
If the terms in the lines of the above display are $r_1$, $r_2$, $r_3$ and $r_4$, it reads $a''=r_1r_2r_3r_4$. Now if we replace $r_1$, $r_2$, $r_3$ and $r_4$ by $r_1(a-1)^{-1}(b-a)$, $(b-a)^{-1}(a-1)r_2(d-1)^{-1}(d-c)$, $(d-c)^{-1}(d-1)r_3(a-1)^{-1}(b-a)$ and $(b-a)^{-1}(a-1)r_4$, their product does not change. On the other hand the new terms admit straightforward cancellations, which leads to
\begin{align*}
a''&=\bigl[b^{-1}(db^{-1}-ca^{-1})^{-1}(ca^{-1}-1)-(b-a)^{-1}(a-1)\bigr]^{-1}
\\
&\quad\times\bigl[(d-b)^{-1}(c-a)-(b-a)^{-1}(a-1)(b-1)^{-1}(b-a)\bigr]
\\
&\quad\times\bigl[(d-c)^{-1}(c-1)(d-1)^{-1}(d-c)-(b-1)^{-1}(b-a)\bigr]^{-1}
\\
&\quad\times\bigl[(d-c)^{-1}(c-1)-(b-a)^{-1}(a-1)\bigr].
\end{align*}
Next, using Lemma~\ref{iden2}, we easily get
\begin{align*}
(b-a)^{-1}(a-1)(b-1)^{-1}(b-a)&=(b-1)^{-1}(b-1)(b-a)^{-1}(a-1)(b-1)^{-1}(b-a)
\\
&=(b-1)^{-1}(a-1)(b-a)^{-1}(b-1)(b-1)^{-1}(b-a)=(b-1)^{-1}(a-1).
\end{align*}
In exactly the same way,
$$
(d-c)^{-1}(c-1)(d-1)^{-1}(d-c)=(d-1)^{-1}(c-1).
$$
Plugging this into the last expression of $a''$ and observing that the last term is exactly $\zeta$ from (\ref{aprpr}), which is invertible according to Lemma~\ref{subma}, we get
\begin{align*}
\zeta a''\zeta^{-1}&=\bigl[(d-c)^{-1}(c-1)-(b-a)^{-1}(a-1)\bigr]\cdot\bigl[b^{-1}(db^{-1}-ca^{-1})^{-1}(ca^{-1}-1)-(b-a)^{-1}(a-1)\bigr]^{-1}
\\
&\quad\times\bigl[(d-b)^{-1}(c-a)-(b-1)^{-1}(a-1)\bigr]\cdot\bigl[(d-1)^{-1}(c-1)-(b-1)^{-1}(b-a)\bigr]^{-1}.
\end{align*}
We deal with the terms $w_1$ and $w_2$  in the two lines of the above display (it reads $\zeta a''\zeta^{-1}=w_1w_2$) separately. Denoting $u=(b-1)^{-1}(a-1)$, $v=(d-1)^{-1}(c-1)$, $d-1=\delta$ and $b-1=\beta$, we have $c-1=\delta v$ and $a-1=\beta u$. In this notation, we have
\begin{align*}
w_2&=\bigl[(d-b)^{-1}(c-a)-(b-1)^{-1}(a-1)\bigr]\cdot\bigl[(d-1)^{-1}(c-1)-(b-1)^{-1}(b-a)\bigr]^{-1}
\\
&=\bigl[(\delta-\beta)^{-1}(\delta v-\beta u)-u\bigr](v-u)^{-1}=(\delta-\beta)^{-1}(\delta v-\beta u)(v-u)^{-1}-
u(v-u)^{-1}
\\
&=(\delta-\beta)^{-1}(\delta v-\beta u)(v-u)^{-1}-u(v-u)^{-1}=
(\delta-\beta)^{-1}(\delta v-\beta v+\beta v-\beta u)(v-u)^{-1}-u(v-u)^{-1}
\\
&=(\delta-\beta)^{-1}(\delta -\beta )v(v-u)^{-1}+(\delta-\beta)^{-1}\beta(v-u)(v-u)^{-1}-u(v-u)^{-1}
\\
&=(v-u)(v-u)^{-1}+(\delta-\beta)^{-1}\beta=1+(\delta-\beta)^{-1}\beta=
(\delta-\beta)^{-1}(\delta-\beta)+(\delta-\beta)^{-1}\beta
\\
&=(\delta-\beta)^{-1}\delta=(d-b)^{-1}(d-1).
\end{align*}
Now we shall demonstrate that $w_1=(d-c)^{-1}(db^{-1}-ca^{-1})b$, that is,
\begin{align*}
&\bigl[(d-c)^{-1}(c-1)-(b-a)^{-1}(a-1)\bigr]\cdot\bigl[b^{-1}(db^{-1}-ca^{-1})^{-1}(ca^{-1}-1)-(b-a)^{-1}(a-1)\bigr]^{-1}
\\
&\qquad=(d-c)^{-1}(db^{-1}-ca^{-1})b.
\end{align*}
Clearly, it is the same as
\begin{align*}
&\qquad(d-c)^{-1}(c-1)-(b-a)^{-1}(a-1)
\\
&=(d-c)^{-1}(db^{-1}-ca^{-1})b\cdot\bigl[b^{-1}(db^{-1}-ca^{-1})^{-1}(ca^{-1}-1)-(b-a)^{-1}(a-1)\bigr].
\end{align*}
Opening the big bracket in the right-hand side and doing straightforward cancellations, we see that the equality $w_1=(d-c)^{-1}(db^{-1}-ca^{-1})b$ is equivalent to
$$
(d-c)^{-1}(c-1)-(b-a)^{-1}(a-1)=(d-c)^{-1}(ca^{-1}-1)-(d-c)^{-1}(db^{-1}-ca^{-1})b(b-a)^{-1}(a-1).
$$
Multiplying by $(d-c)$ from the left and cancelling out the term $-1$, we arrive to
$$
c-(d-c)(b-a)^{-1}(a-1)=ca^{-1}-(db^{-1}-ca^{-1})b(b-a)^{-1}(a-1).
$$
Opening up the brackets, we rewrite this equation as
$$
c-d(b-a)^{-1}(a-1)+c(b-a)^{-1}(a-1)=ca^{-1}-d(b-a)^{-1}(a-1)+ca^{-1}b(b-a)^{-1}(a-1).
$$
After an obvious cancellation and after multiplying by $ac^{-1}$ from the left, we get:
\begin{align*}
&w_1=(d-c)^{-1}(db^{-1}-ca^{-1})b\iff a+a(b-a)^{-1}(a-1)=1+b(b-a)^{-1}(a-1)
\\
&\iff a-1=b(b-a)^{-1}(a-1)-a(b-a)^{-1}(a-1)\iff a-1=(b-a)(b-a)^{-1}(a-1).
\end{align*}
Since the last equality is obviously true, this proves that $w_1=(d-c)^{-1}(db^{-1}-ca^{-1})b$. Since $\zeta a''\zeta^{-1}=w_1w_2$ and $w_2=(d-b)^{-1}(d-1)$, we arrive to
$$
\zeta a''\zeta^{-1}=(d-c)^{-1}(db^{-1}-ca^{-1})b(d-b)^{-1}(d-1)=(d-c)^{-1}(db^{-1}-ca^{-1})(db^{-1}-1)^{-1}(d-1)
$$
and (\ref{aprpr}) follows.
\end{proof}

We also need a presentation for $\Psi=J_2\circ\Phi\circ J_2$, defined earlier.

\begin{lemma}\label{PHI2EEE} Let
\begin{equation*}
A= \left(\begin{array}{ccc}
1&1&1\\
1&a&b\\
1&c&d
\end{array}\right)\in \widehat{\cal S}\ \ \text{and}\ \ \Psi(A)=\left(\begin{array}{ccc}
1&1&1\\
1&a^+&b^+\\
1&c^+&d^+
\end{array}\right).
\end{equation*}
Then
\begin{equation}\label{aprpri}
a^+=(d-c)^{-1}(db^{-1}-ca^{-1})(db^{-1}-1)^{-1}(d-1).
\end{equation}
 \end{lemma}

 \begin{proof} As $\Psi=J_2\circ\Phi\circ J_2$, an expression for $a^+$ is obtained from the expression (\ref{hfex}) for the corresponding entry $a'$ by replacing $(a,b,c,d)$ with $(a^{-1},c^{-1},b^{-1},d^{-1})$ and inverting the result:
 \begin{align*}
 a^+&=\bigl[(d^{-1}-1)^{-1}(d^{-1}-b^{-1})a(d^{-1}c-b^{-1}a)^{-1}(d^{-1}c-1)\bigr]^{-1}
 \\
 &=(d^{-1}c-1)^{-1}(d^{-1}c-b^{-1}a)a^{-1}(d^{-1}-b^{-1})^{-1}(d^{-1}-1)
 \\
 &=(c-d)^{-1}d(d^{-1}c-b^{-1}a)a^{-1}b(b-d)^{-1}d(d^{-1}-1)
 \\
 &=(c-d)^{-1}(ca^{-1}-db^{-1})b(b-d)^{-1}(1-d)
 \\
 &=(d-c)^{-1}(db^{-1}-ca^{-1})(db^{-1}-1)^{-1}(d-1)
 \end{align*}
 as required.
\end{proof}

\subsection{Proof of Theorem~\ref{main2}}

Let
$$
A= \left(\begin{array}{ccc}
1&1&1\\
1&a&b\\
1&c&d
\end{array}\right)\in \widehat{\cal S},\quad \Psi(A)=\left(\begin{array}{ccc}
1&1&1\\
1&a^+&b^+\\
1&c^+&d^+
\end{array}\right),\ \ \text{and}\ \ \Phi^2(A)=\left(\begin{array}{ccc}
1&1&1\\
1&a''&b''\\
1&c''&d''
\end{array}\right).
$$
Denote
$$
\zeta(a,b,c,d)=(d-c)^{-1}(c-1)-(b-a)^{-1}(a-1).
$$
By Lemmas~\ref{aprpri} and~\ref{PHI2E}, $a''=\zeta(a,b,c,d)^{-1}a^+\zeta(a,b,c,d)$.
Corollary~\ref{kley1} (commuting with the Klein 4-group action) yields
\begin{equation}\label{AAAA}
\begin{array}{ll}
a''=\zeta(a,b,c,d)^{-1}a^+\zeta(a,b,c,d);\quad&b''=\zeta(b,a,d,c)^{-1}b^+\zeta(b,a,d,c);
\\
c''=\zeta(c,d,a,b)^{-1}c^+\zeta(c,d,a,b);&d''=\zeta(d,c,b,a)^{-1}d^+\zeta(d,c,b,a).
\end{array}
\end{equation}
We claim that
\begin{equation}\label{BBBB}
\zeta(a,b,c,d)=\zeta(d,c,b,a)=-\zeta(b,a,d,c)=-\zeta(c,d,a,b).
\end{equation}
Indeed, observe that
$$
(d-c)^{-1}(c-1)+1=(d-c)^{-1}(c-1)+(d-c)^{-1}(d-c)=(d-c)^{-1}(c-1+d-c)=(d-c)^{-1}(d-1).
$$
Similarly, $(b-a)^{-1}(a-1)+1=(b-a)^{-1}(b-1)$. and therefore
\begin{equation}\label{POO}
\zeta(a,b,c,d)=(d-c)^{-1}(c-1)-(b-a)^{-1}(a-1)=(d-c)^{-1}(d-1)-(b-a)^{-1}(b-1).
\end{equation}
Now, using (\ref{POO}), we obtain
\begin{align*}
\zeta(d,c,b,a)&=(a-b)^{-1}(b-1)-(c-d)^{-1}(d-1)=(d-c)^{-1}(d-1)-(b-a)^{-1}(b-1)=\zeta(a,b,c,d),
\\
\zeta(c,d,a,b)&=(b-a)^{-1}(a-1)-(d-c)^{-1}(c-1)=-\zeta(a,b,c,d),
\\
\zeta(b,a,d,c)&=(c-d)^{-1}(d-1)-(a-b)^{-1}(b-1)=-\zeta(a,b,c,d).
\end{align*}
Thus (\ref{BBBB}) is satisfied. Combining it with (\ref{AAAA}) and calling $\zeta(a,b,c,d)$ by just $\zeta$, we obtain
$$
\Phi^2(A)=\zeta^{-1}\Psi(A)\zeta.
$$
By Proposition~\ref{ph-1m}, there is $\xi=\xi(A)\in R^*$ such that
$$
\Psi(A)=\xi^{-1}\Phi^{-1}(A)\xi.
$$
The last two displays immediately imply Theorem~\ref{main2}.

\subsection{An explicit formula for the conjugating element}

Let
$$
A= \left(\begin{array}{ccc}
1&1&1\\
1&a&b\\
1&c&d
\end{array}\right)\in \widehat{\cal S}
$$
The already proved Theorem~\ref{main2} asserts the existence of $\omega=\omega(A)\in R^*$ such that $\Phi^2(A)=\omega^{-1}\Phi^{-1}(A)\omega$. In the next theorem we give an explicit formula for this conjugating element.

\begin{theorem} The conjugating element $\omega$, s.t.
$\Phi^2(A)=\omega^{-1}\Phi^{-1}(A)\omega$ is the following:

$$
\omega=\bigl[d(d-b)^{-1}-c(c-a)^{-1}\bigr]^{-1}\cdot
\bigl[d(d-b)^{-1}(b-1)-c(c-a)^{-1}(a-1)\bigr]\cdot\bigl[(d-c)^{-1}(c-1)-(b-a)^{-1}(a-1)\bigr].
$$

\end{theorem}

\begin{proof} According to our choice of $\Phi$, discussed above we have:
$$\Phi(A)= J_2 \Lambda^L J_1(A),$$
 where $ \Lambda^L(B)=D_1^{-1}BD_2$ is the matrix from  $\widehat M_3(R)$, $D_1$, and $D_2$ are invertible diagonal matrices and $D_1$ has $1$ in the left upper corner. In the Lemma \ref{PHI2E} we have proved that
 $\Phi^2(A)$ is conjugate to the transformation $\Psi$ obtained from $\Phi$ via composition conjugation by $J_2$: $$\Psi = J_2 \circ \Phi \circ J_2.$$ The conjugating element $\zeta$, such that  $$\Phi^2(A)=\zeta^{-1}\Psi(A)\zeta,$$ is the following:  $\zeta=(d-c)^{-1}(c-1)-(b-a)^{-1}(a-1).$

Now we need to find a conjugating element $\nu$, such that
$$\Psi(A) = \nu^{-1} \Phi^{-1}(A)\nu.$$
Then we would have
$$\Phi^2(A)=\zeta^{-1}\nu^{-1}\Phi^{-1}(A) \nu\zeta,$$
so $\omega=\nu\zeta.$

To find a conjugating element of $\Psi(A)$ and $\Phi^{-1}(A)$, we should  remember how the inverse map for $\Psi$ was defined (see Proposition\ref{phiinve}):
$$\Phi^{-1}(A)= \Lambda^R \circ J_1 \circ J_2(A)=\Lambda^R \circ J^{-1} (A).$$
where $\Lambda^R = D_3^{-1}BD_4^{-1}$ belongs to $\widehat M_3(R)$, $D_3$ and $D_4$ are invertible diagonal matrices, and $D_3$ has $1$ in the left upper corner.

Since
$$\Psi(A)= J_2 \circ J_2 \circ \Lambda^L \circ J_1 \circ J_2(A)=\Lambda^L \circ J_1 \circ J_2(A)$$
and
$$\Phi(A)=\Lambda^R \circ J_1 \circ J_2(A)$$
we can write down precisely diagonal matrices featuring in $\Lambda^L$ and $\Lambda^R$ and compare element wise the resulting matrices for $\Psi$ and $\Phi^{-1}$:

$$\Psi^{-1}(A)=D_1^{-1}(J_1( J_2(A))D_2$$
and
$$\Phi^{-1}(A)=D_3^{-1}(J_1( J_2(A))D_4,$$
where
$D_1,D_2,D_3$ and $D_4$ are invertible diagonal matrices, with $D_2$ and $D_3$ having $1$ at the left upper corner.

Hence,
$$\Psi(A)=D_3^{-1}D_1 \Phi^{-1}(A) D_2^{-1}D_4.$$

According to the Conjugacy lemma \ref{nonu}, $D_2^{-1}D_4=D_3^{-1}D_1=x I$ for some element $x \in R^*$. We know that the first diagonal element in $D_2$ is $1$, hence for $D_2^{-1}D_4=x I$ to be true, the first diagonal element in $D_4$ should be $x$.

Now using the latter formula for $\Psi(A)$, and comparing left upper corner elements in its left and right-hand side, we get: $1=1 \cdot y \cdot x$, where $y$ is the left upper corner of the matrix

$$J_2(A)^{-1} = \left(\begin{array}{ccc}
1&1&1\\
1&a^{1}&c^{-1}\\
1&b^{1}&d^{1}
\end{array}\right)^{-1}
$$

Thus $x=y^{-1}$, and it could be calculated using Lemma \ref{inve}. We get then
$$\nu=x=\bigl[d(d-b)^{-1}-c(c-a)^{-1}\bigr]^{-1}\cdot
\bigl[d(d-b)^{-1}(b-1)-c(c-a)^{-1}(a-1)\bigr]$$
and the formula for $\omega$ follows from that, taking into account the formula for $\zeta$.

\end{proof}

Although probably, $\omega$ can be written in a somewhat shorter way, we believe that an alternative expression for $\omega$ can not be much shorter. To support this statement we note that the last term in $\omega$ is the one responsible for the invertibility of $A$ according to Lemma~\ref{subma} (it is a version of a noncommutative determinant), while the middle term plays the same role for $J_2(A)$. Since the invertibility of $A$ and $J_2(A)$ do not imply each other even in the commutative case, one can not expect any dramatic cancellations between the two biggest terms in the expression for $\omega$.

\section{Identities  for the Hadamard product}

We derive here certain particular identities in terms of the Hadamard product of matrices, in certain subsets of all matrices, where our transformations acts, such as matrices with commuting entries, or subset $\cal M$ from the statement below: matrices for which $a=c, b=d$.
These identities might be particular instances of some general noncommutative identity, which we do not know.

\begin{theorem} IF $R$ is commutative, then the identity

$$\Phi^{0}(A) \star \Phi^{1}(A) \star \Phi^{2}(A) =\Phi^{0}(A) \star \Phi^{1}(A) \star \Phi^{-1}(A) = {\bf 1},$$

holds on $\widehat M_3(R) \cap {\rm dom}(\Phi^2)$,
where $\star$ stands for the Hadamard (componentwise) product of matrices, and ${\bf 1}$   is the matrix with all entries equal to $1.$

\end{theorem}

Another particular case, apart from matrices on commuting elements, could be considered. Namely,
due to the symmetry lemma the subset of matrices $\cal M$ of the shape

$$\left(\begin{array}{ccc}
1&1&1\\
1&a&b\\
1&b&a
\end{array}\right)
$$

is closed under the operation $\Phi$. Thus on this subset we can derive the following identity.

\begin{theorem} On the subset
$${\cal M} = \{\left(\begin{array}{ccc}
1&1&1\\
1&a&b\\
1&b&a
\end{array}\right)\}
$$

the following identity holds:

$$\Phi^{0}(\xi) \star \Phi^{-1}(\xi) \star \Phi^{1}(\xi)  = {\bf 1}.$$

\end{theorem}

\section{Acknowledgements}

 We would like to thank Maxim Kontsevich for encouragement and attention to this work.
 We are also grateful to IHES and MPIM for hospitality, support, and excellent research atmosphere.
 This work is funded by the ERC grant 320974.

\vspace{77mm}

\noindent   Natalia Iyudu

\noindent School of Mathematics

\noindent  The University of Edinburgh

\noindent James Clerk Maxwell Building

\noindent The King's Buildings

\noindent Mayfield Road

\noindent Edinburgh

\noindent Scotland EH9 3JZ

\noindent E-mail address: \qquad {\tt niyudu@staffmail.ed.ac.uk}\ \ \

\vskip 11mm

\noindent   Stanislav Shkarin

\noindent Queens's University Belfast

\noindent Department of Pure Mathematics

\noindent University road, Belfast, BT7 1NN, UK

\noindent E-mail address: \qquad {\tt s.shkarin@qub.ac.uk}

\vskip 5mm

\end{document}